\documentclass[11pt,reqno,twoside]{amsart}
\usepackage{amssymb,amsfonts,amsthm,amsmath,mathrsfs}  
\usepackage{color}
\usepackage{hyperref}
\usepackage[left=1 in,top=1 in,right=1 in,bottom=1 in]{geometry}
\usepackage{cite}
\usepackage{enumitem}
\usepackage{mathtools}
\mathtoolsset{showonlyrefs}

\newcommand{\beq}{\begin{equation}}
\newcommand{\eeq}{\end{equation}}
\newcommand{\beqs}{\begin{equation*}}
\newcommand{\eeqs}{\end{equation*}}
\newcommand{\ba}{\begin{array}}
\newcommand{\ea}{\end{array}}
\newcommand{\beas}{\begin{eqnarray*}}
\newcommand{\eeas}{\end{eqnarray*}}
\newcommand{\bea}{\begin{eqnarray}}
\newcommand{\eea}{\end{eqnarray}}
\newcommand{\bal}{\begin{align}}
\newcommand{\eal}{\end{align}}

\newcommand{\bals}{\begin{align*}}
\newcommand{\eals}{\end{align*}}

\newcommand{\R}{\ensuremath{\mathbb R}}

\newcommand{\m}{\mathbf m }
\newcommand{\vv}{\mathbf v}
\newcommand{\ww}{\mathbf w}
\newcommand{\uu}{\mathbf u}

\newcommand{\intb}[1]{\left\langle #1 \right\rangle}
\newcommand{\diver}{{\rm {div} \, }}
\newcommand{\inprod}[1]{\langle{#1}\rangle}

\newcommand{\norm}[1]{\| {#1} \|}

\newcommand{\bds}{\begin{displaystyle}}
\newcommand{\eds}{\end{displaystyle}}

\def\eqdef{\stackrel{\rm def}{=}}

\def\varep{\varepsilon}

\def\d{{\rm d}}

\newcommand{\tnum}{\rm(\roman*)}
\newcommand{\rnum}{\rm(\alph*)}

\newtheorem{theorem}{Theorem}[section]

\newtheorem{lemma}[theorem]{Lemma}
\newtheorem{corollary}[theorem]{Corollary}
\newtheorem{proposition}[theorem]{Proposition}
\newtheorem{definition}[theorem]{Definition}
\newtheorem{assumption}[theorem]{Assumption}

\theoremstyle{remark}
\newtheorem{example}[theorem]{\bf{Example}}
\newtheorem{remark}[theorem]{\bf{Remark}}

\numberwithin{equation}{section}
\date{\today}

\title[Anisotropic Flows of Forchheimer-type in  Porous Media]
{Anisotropic flows of Forchheimer-type in porous media and their steady states}

\author[L. Hoang]{Luan Hoang$^{1,*}$}
\address{$^1$Department of Mathematics and Statistics,
Texas Tech University\\
1108 Memorial Circle, Lubbock, TX 79409--1042, U. S. A.}
\email{luan.hoang@ttu.edu}

\author{Thinh Kieu$^{2}$}
\address{$^{2}$Department of Mathematics, University of North Georgia, Gainesville Campus\\
3820 Mundy Mill Rd., Oakwood, GA 30566, U. S. A.}
\email{thinh.kieu@ung.edu}

\thanks{$^*$Corresponding author.}


\makeatletter
\@namedef{subjclassname@2020}{\textup{2020} Mathematics Subject Classification}
\makeatother

\begin{document}

\begin{abstract} 
We study the anisotropic Forchheimer-typed flows for compressible fluids in porous media.
The first half of the paper is devoted to understanding the nonlinear structure of the anisotropic momentum  equations. Unlike the isotropic flows, the important monotonicity properties are not automatically satisfied in this case. Therefore, various sufficient conditions for them are derived and applied to the experimental data. 
In the second half of the paper, we prove the existence and uniqueness of the steady state flows subject to a nonhomogeneous Dirichlet boundary  condition. It is also established that these steady states, in appropriate functional spaces, have local H\"older continuous dependence on the forcing function and the boundary data.
\end{abstract}

\maketitle 

\tableofcontents 

\pagestyle{myheadings}\markboth{{\sc L. Hoang and T. Kieu}}{{\sc Anisotropic Flows of Forchheimer-type in  Porous Media}}

\section{Introduction}\label{Intro}
We study fluid flows in porous media. In this section,  
$p\in \R$ is the pressure, 
$\vv\in \R^3$ is the velocity, 
$\mu>0$ is the dynamic  viscosity, $\rho\ge 0$ is the density, 
$k>0$ is the permeability, 
and $\phi\in(0,1)$ is the constant porosity.
The standard momentum equation for fluid flows in porous media is the Darcy's law
\begin{equation}\label{Darcy}
-\nabla p = \frac {\mu}{k} \vv.
\end{equation}
However, it was noticed very early by Dupuit \cite{Dupuit1857} that  there was a deviation from this Darcy's law. 
Nowadays, Darcy's law is used widely for flows with low Reynolds numbers.
For higher Reynolds numbers, Forchheimer \cite{Forchh1901,ForchheimerBook} suggested three empirical nonlinear models called the two-term, three-term and power laws.
The Forchheimer two-term  law is
\begin{equation}\label{FL2}  
a_0 \vv+a_1|\vv|\vv=-\nabla p, \text{ where }a_0,a_1>0.
\end{equation}
The Forchheimer three-term  law is
\begin{equation}\label{FL3}
a_0 \vv+a_1|\vv|\vv+a_2|\vv|^2\vv=-\nabla p, \text{ where }a_0,a_1,a_2>0.
\end{equation}
The Forchheimer power law is
\begin{equation}\label{FLP}  a_0 \vv+a_1|\vv|^{m-1}\vv=-\nabla p, \text{ where  $a_0,a_1>0$ and $1.6\le m\le 2$.}
\end{equation}
The nonlinear terms in \eqref{FL2}, \eqref{FL3} and \eqref{FLP} were added to correct the deviation from the Darcy law \eqref{Darcy} and fit the experimental data.
The two-term law \eqref{FL2} was later refined and confirmed by experiments by Ward \cite{Ward64}, using the dimension analysis known in Muskat's book \cite{MuskatBook}, to become
\begin{equation}\label{FW}  
\frac{\mu}{k} \vv+c_F\frac{\rho}{\sqrt k}|\vv|\vv=-\nabla p, \text{ where }c_F>0.
\end{equation}
There are many other nonlinear models throughout the years such as Ergun \cite{Ergun1952}, Scheidegger \cite{Scheidegger1974}, Irmay \cite{Irmay1964d}, Ahmed \cite{Ahmed1967}, etc., see  \cite[Section 5.11.3]{BearBook} for more equations.
For example, Scheidegger \cite[formulas (7.2.3.9) and  (7.2.3.14)]{Scheidegger1974} obtained equation \eqref{FL2} with
\beqs
a_0=C_1\frac{\mu T^2}{\phi}\text{ and } a_1=C_2\frac{\rho T^3}{\phi^2},
\eeqs 
where $C_1,C_2$ are coefficients depending on the grain-size distribution, and $T$ is tortuosity.
Despite of having more specific coefficients, these nonlinear equations are similar to the three Forchheimer laws above.
The Forchheimer-typed equations can also be derived theoretically in various ways, see e.g. \cite{Irmay1958,ChenLyonsQin2001,ELLP,Whitaker1996,Scheidegger1974}.
For more information about Forchheimer flows in porous media, see e.g. \cite{BearBook,MuskatBook,NieldBook,Scheidegger1974}.

The generalized Forchheimer equations were proposed to unify all the known forms of the Forchheimer-typed equations.  They are of the form
\beq\label{genF}
\sum_{i=0}^m a_i |\vv|^{\alpha_i}\vv=-\nabla p, 
\eeq  
where $a_i$'s are positive numbers and  $0=\alpha_0<\alpha_1<\ldots<\alpha_m$.

All of the above equations are written for isotropic porous media.
For anisotropic porous media, Bachmat \cite{Bachmat1965} gave an equation  
$$ a_0\vv+a_1|\vv|\vv=-K\nabla p,\text{ where $K$ is the permeability tensor.}$$
Barak and Bear \cite{BarakBear81} later formulated many physical and mathematical models for  fluid flows at high Reynolds numbers in anisotropic porous media. 
We review here one of those physical models (the first one).
Denote by $\mathbf q$, $\mathbf J$, $g$, $\nu$ the specific discharge vector, the hydraulic gradient, the gravity acceleration and kinematic viscosity.
Consider a porous medium that contains three bundles of pipes which are parallel to each other and have the same diameters within each bundle, and the bundles' directions are orthogonal to each other.
For the $\lambda$th group, $\lambda=1,2,3$, its axis is parallel to the unit vector $\mathbf 1_\lambda$, and let $M_\lambda$, $D_{c\lambda}$, $A_{c\lambda}$ be the number of pipes per unit area normal to their axis, the pipes' diameter, the average cross-sectional area of a pipe in the group. Assume that the frictions in two opposite directions along a pipe are the same $C_c$ which yields $C_\Delta=0$ and $C_\lambda=0$ in formulas (8)--(14) of \cite{BarakBear81}.
Then the derived equation (13) in \cite{BarakBear81} reads as
\beq \label{qorig}
\mathbf q\cdot(\nu \mathbf w^{(2)}+|\mathbf q|\mathbf B^{(2)}\cdot\mathbf Q^{(2)})=g\mathbf J,
\eeq 
where $\mathbf w^{(2)}$, $\mathbf B^{(2)}$, $\mathbf Q^{(2)}$ are tensors of 2nd order defined by 
$$\mathbf w^{(2)}=\sum_{\lambda=1}^3 w_\lambda \mathbf 1_\lambda\mathbf 1_\lambda,\quad 
\mathbf B^{(2)}=\sum_{\lambda=1}^3 B_\lambda \mathbf 1_\lambda\mathbf 1_\lambda, \quad 
\mathbf Q^{(2)}=\sum_{\lambda=1}^3 |\cos\,\varphi_\lambda|\mathbf 1_\lambda\mathbf 1_\lambda, $$
with 
$w_\lambda={C_0}/(A_{c\lambda}D_{c\lambda}^2M_\lambda)$, $B_\lambda={C_c}/(A_{c\lambda}^2D_{c\lambda}^2M_\lambda^2)$, 
$C_0$ being a dimensionless coefficient, and 
$\varphi_\lambda$ being the angle between $\mathbf q$ and $\mathbf 1_\lambda$.

In addition to this physical model and \eqref{qorig}, \cite{BarakBear81} considers four more cases which yield corresponding equations.
Based on those, many more mathematical models and generalizations were investigated. In particular, one mathematical generalization of \eqref{qorig} is
\beq \label{qorig2}
\nu \mathbf w^{(2)}\cdot \mathbf q +\frac1{|\mathbf q|}\mathbf q\mathbf q:\mathbf B^{(4)}\cdot \mathbf q=g\mathbf J,
\eeq 
where $\mathbf B^{(4)}$ is a 4th order tensor.
However, experiments reported in \cite{BarakBear81} yield only one equation, namely,
\beq\label{BB0}
\mu A_0 \vv+\frac\rho{|\vv|}\begin{pmatrix}
       (av_1^2+bv_2^2)v_1\\ (cv_1^2+dv_2^2)v_2
      \end{pmatrix} =-\nabla p\text{ for } \vv=(v_1,v_2)\in\R^2,
\eeq
where $A_0={\rm diag}[109,220]$ - a diagonal matrix, and $a,b,c,d$ are the following positive constants
\beq\label{abcd}
a=0.20,\quad b=1.04,\quad c=0.67,\quad d=1.15.
\eeq
Although more complicated models \cite{Cvetko86,HassanGray87} were proposed, we will use equation \eqref{BB0} as our motivation because it was actually verified by experiments.

Equation \eqref{BB0} can be formulated in a general form, taking into account the forms of equations \eqref{qorig} and \eqref{qorig2}, as
\beq \label{BB}
A_0 \vv + |\vv| A\vv +\frac1{|\vv|} B(\vv,\vv,\vv) =-\nabla p,
\eeq
where  $A_0$ and $A$  are square matrices, and $B(\cdot,\cdot,\cdot)$  is a trilinear mapping.

The first term in \eqref{BB} is the anisotropic Darcy term which corresponds to the first terms in \eqref{qorig}, \eqref{qorig2} and \eqref{BB0}.
The second term in \eqref{BB} corresponds to that of \eqref{qorig}. It is  an anisotropic version of the drag term $a_1|\vv|\vv$ in \eqref{FL2}.
The third  term in \eqref{BB} corresponds to the second terms in \eqref{qorig2} and \eqref{BB0}.
Mathematically speaking, the third term in \eqref{BB} is a generalization of the second term. Indeed, if 
$$B(\uu,\vv,\ww)=(\uu\cdot\vv)A\ww\text{ then }\frac1{|\vv|}B(\vv,\vv,\vv)=|\vv|A\vv.$$
However, we keep both of these terms in \eqref{BB} because they will have very different mathematical treatments later.

While the mathematical study of the Darcy law has a long history and vast literature, see the treaty \cite{VazquezPorousBook} and references there in, the isotropic Forchheimer equations were studied much later, see e.g. \cite{Fabrie1989a,Franchi2003,Payne1996,Payne1999b,Payne2000a,Payne2000b,ChadamQin,StraughanBook} for incompressible fluids and \cite{ABHI1,HI2,HK2,HKP1,HIKS1,CHK2,CH2,CHK4} for compressible fluids. To the best of our knowledge, the current paper is the first to study rigorously the mathematics of anisotropic Forchheimer equation \eqref{BB}.

When the dependence on the density is needed, we can use the dimension analysis in \cite{MuskatBook}, to modify equation \eqref{BB}. Same as equation \eqref{FW}, it results in the following equation
\beq \label{AAB}
A_0 \vv + \rho |\vv| A\vv +\frac\rho {|\vv|} B(\vv,\vv,\vv) =-\nabla p.
\eeq

The momentum equation \eqref{AAB} is coupled with the conservation of mass which is 
\beq\label{mass}
\phi \rho_t + \nabla \cdot (\rho \vv)=f,
\eeq
where $f(x,t)$ is the source/sink term.

Multiplying equation \eqref{AAB} by $\rho$ gives
\beq\label{rveq}
A_0(\rho \vv)+|\rho \vv|A(\rho \vv) + \frac{1}{|\rho \vv|}B(\rho \vv,\rho \vv,\rho \vv)=-\rho \nabla p.
\eeq

Denote the momentum by $\m=\rho \vv$. Defining $F(\vv)$ to be the left-hand side of \eqref{BB},  we then obtain from \eqref{rveq} and \eqref{mass} that
\beq \label{msys}
F(\m)=-\rho \nabla p,\quad 
\phi \rho_t + \nabla \cdot \m=f.
\eeq 

\medskip\noindent
\emph{The case of isentropic flows.} We have $p=\bar c\rho^\gamma$ with $\bar c,\gamma>0$.
Set 
\beq
\widetilde p=\frac{\bar c\gamma}{\gamma+1} \rho^{\gamma+1},\quad \lambda=\frac{1}{\gamma+1},
\quad c=\left(\frac{\gamma+1}{\bar c\gamma}\right)^\lambda.
\eeq
Then the system \eqref{msys} becomes
\beq\label{pseu}
F(\m)=- \nabla \widetilde p,\quad 
c \phi (\widetilde p^\lambda)_t + \nabla \cdot \m=f.
\eeq 
In particular, $\gamma=1$ and $\lambda=1/2$ for ideal fluids. The quantity $\widetilde p$, up to a constant multiplier, is known as the pseudo pressure.

\medskip\noindent
\emph{The case of slightly compressible flows.}  We have that the compressibility 
\beq
\kappa:=\frac1\rho\frac{\d \rho}{\d p} \text{ is a small, positive constant.}
\eeq
Observe that the first equation in \eqref{msys} now is 
$F(\m)=-\kappa^{-1} \nabla \rho$.
By setting 
$$\widetilde p=\rho /\kappa,\quad \lambda=1\text{ and }c=\kappa,$$
we obtain the system \eqref{pseu} again from \eqref{msys}

\medskip
Our next task is to analyze \eqref{pseu} which is a  system of nonlinear partial differential equations (PDE).
First, we study the nonlinear structure of the momentum equation \eqref{AAB} and its generalizations. The key properties that we investigate are the monotonicities for the function $F$, see Definition \ref{Mdef}. Because of the counter examples \ref{nomono} and \ref{nomoF}, many sufficient conditions need be developed. Once the monotonicities are obtained, we can solve for $\m=-X(\nabla \widetilde p)$ from the first equation of \eqref{pseu} and substitute it into the second equation. It results in a second order scalar parabolic equation\beq
c\phi  \frac{\partial (\widetilde p^\lambda)}{\partial t}+\nabla\cdot (X(\nabla \widetilde p))=f.
\eeq

This approach was widely exploited in our previous work, see \cite{HI2,HK2,HKP1,HIKS1,CHK2,CHK4,CH1,CH2} and references therein, and also \cite{GM1975,Chow1989,BN1990,Sandri1993,Fabrie1989a} for the stationary problems and their numerical analysis. 
In the current work, we follow a more straightforward approach, namely, studying system \eqref{pseu} directly as a system of the first order PDE. 
As a demonstration, we will focus on the steady state flows with a nonhomogeneous Dirichlet boundary condition. 
We obtain the existence and uniqueness of these anisotropic flows for any order $s>2$ in Assumption \ref{assuF}. They are more general than the counterpart in \cite{KS16} for the isotropic Forcheimer flows corresponding to $s=3$, as well as \cite{Fabrie1989a} for isotropic Forchheimer's power laws with  the flux  boundary condition. 
Moreover, we establish explicit estimates and continuous dependence results which are not obtained in \cite{Fabrie1989a,KS16}.

The paper is organized as follows.
In Section \ref{prelim}, we present inequalities that will be used throughout the paper. Also, relevant functional spaces and their  important properties are recalled. 
In Section \ref{mod1sec}, we study properties of the function $G_B$ defined by \eqref{Bnd} and \eqref{GBdef}. It  is the $n$-dimensional version of the last term on the left-hand side of \eqref{BB}. 
Their upper and lower bounds, and Lipschitz estimates are obtained in Lemma \ref{Bprop}.
For our mathematical treatments, the monotonicities in Definition \ref{Mdef} are crucial.
However, such monoticities are not automatic under the general positivity condition for the coefficients, see Example \ref{nomono}. Therefore, the large part of this section is devoted to finding sufficient conditions for the monotonicities. Subsection \ref{ss2d} treats the two-dimensional case, while subsection \ref{ssnd} does the higher dimensional case. Our sufficient conditions turn out to be applicable to the real data obtained in \cite{BarakBear81}, see Example \ref{eg2}. 
In Section \ref{mod2sec}, we consider other possible anisotropic flows in the form of the functions $F_{A,\alpha}$ and $\widetilde F_{A,\alpha}$, see \eqref{FAdef} and \eqref{tildeFA}. These represent different anisotropic versions of \eqref{genF}. While $\widetilde F_{A,\alpha}$ is $(\alpha+2)$-monotone under the coercivity condition \eqref{Avv}, the function $F_{A,\alpha}$ may not be so, see Example \ref{nomoF}.
Therefore, many sufficient conditions for the monotonicities are derived in Theorems \ref{Monothm7}, \ref{Monothm5} and \ref{Monothm6}. At the end of the section, many scenarios for the momentum equations are shown in Proposition \ref{thmix}.
In Section \ref{StatProb}, we study the steady state flows for a general system of the form \eqref{pseu} subject to a time-independent nonhomogeneous Dirichlet boundary condition, see Problem \eqref{stationaryProb}.
Thanks to our understanding in Sections \ref{mod1sec} and \ref{mod2sec}, natural and general conditions on the function $F$ are imposed in Assumption \ref{assuF}. 
The variational formulation in suitable functional spaces for the weak solutions is established, see Definition \ref{weaksoln}. 
Theorem \ref{mainthm} contains the main results for system \eqref{stationaryProb} including the existence, uniqueness of the weak solutions, together with their estimates and continuous dependence on the forcing function and boundary data. 
For its proof, we generalize the method used in \cite{KS16}.  
In fact, a general functional equation is studied instead, see \eqref{VQgen}.
Its regularized problem is introduced and investigated in subsection \ref{subreg}.
Using this approximation, we obtain the existence and uniqueness in Theorem \ref{SolofStationaryProb}. Estimates of the solutions are obtained in Theorem \ref{postest} and the continuous dependence result is established in Theorem \ref{cdthm}.
Gathering the above general results gives a quick proof of Theorem \ref{mainthm} in subsection \ref{mainpf}.
Section \ref{conclude} contains concluding remarks and further discussions.

\section{Preliminaries}\label{prelim}

\subsection{Notation and inequalities} 
For a vector $x\in\R^n$, its Euclidean norm is denoted by $|x|$.

Let $ A=(a_{ij})_{1\le i,j\le n}$  be an $n\times n$ matrices of real numbers. The Euclidean norm of $ A$ is $| A|=\Big(\sum_{i,j=1}^n a_{ij}^2\Big)^{1/2}.$
(We do not use $| A|$ to denote the determinant in this paper.)
When $ A$ is considered as a linear operator, another norm is defined by
\beqs
\| A\|_{\rm op}=\max\left\{ \frac{| Ax|}{|x|}:x\in\R^n,x\ne 0\right\} = \max\{ | Ax|:x\in\R^n,|x|=1\}.
\eeqs
It is well-known that
\beq\label{nnorms}
\| A\|_{\rm op}\le | A|\le \bar c_0\| A\|_{\rm op},
\eeq
where $\bar c_0=\bar c_0(n)$ is a positive constant independent of $ A$.

For any $x\in\R^n$, we have
\beq\label{op1}
|Ax|\le \|A\|_{\rm op}|x|\le |A|\,|x|,
\eeq 
and, in the case $A$ is invertible,
\beq \label{op2}
|Ax|\ge \|A^{-1}\|_{\rm op}^{-1}|x|.
\eeq 

The following are some commonly used consequences of Young's inequality. If $x,y\ge 0$, $\gamma\ge\beta\ge \alpha>0$, $p,q>1$ with $1/p+1/q=1$, and $\varep>0$, then
\beq\label{Yineq} 
x^\alpha\le 1+x^\beta, \quad x^\beta\le x^\alpha+x^\gamma, \quad
xy\le \varep x^p + \varep^{-q/p}y^{q}. 
\eeq 

For $z\in \R$, denote $z^+=\max\{0,z\}$.
For $x,y\in \R^n$ and $p>0$, one has
\beq \label{eleori}
(|x|+|y|)^p\le 2^{(p-1)^+}(|x|^p+|y|^p)
=\begin{cases} |x|^p+|y|^p,& \text{ for }p\in(0,1],\\
 2^{p-1}(|x|^p+|y|^p),& \text{ for }p>1,
 \end{cases}
\eeq 
which  consequently yields
\begin{align} \label{ele0}
(|x|+|y|)^p&\le 2^p (|x|^p+|y|^p)&& \text{ for }p>0,\\
 \label{ele1}
||x|^p-|y|^p| &\le |x-y|^p&&\text{ for } p\in(0,1].
\end{align} 
Regarding the last inequality, one has, in the case $p>1$, that
 \beq\label{ele2} 
    ||x|^p -|y|^p |\le 2^{(p-2)^+}(|x|^{p-1}+|y|^{p-1})|x-y|.
    \eeq

\begin{lemma}\label{ILem1}
If $p>0$ and $x,y\in\R^n$, then
\begin{align}\label{inq1}
||x|^p x-|y|^p y|&\le 2^{(p-1)^+} (|x|^p+|y|^p)|x-y|,\\
\label{inq3m}
(|x|^p x-|y|^p y)\cdot (x-y)&\ge \frac{1}{2}(|x|^p+|y|^p)|x-y|^{2},\\
\label{inq3}
(|x|^p x-|y|^p y)\cdot (x-y)&\ge \frac{1}{2^{1+(p-1)^+}}|x-y|^{p+2}.
\end{align}

If $p\in (-1,0)$, then
\begin{align} \label{inq2}
||x|^{p} x-|y|^{p} y| &\le 2^{-p}|x-y|^{1+p} &&\text{ for all $x,y\in \R$,}\\
\label{inq4}
(|x|^{p}x-|y|^{p} y)\cdot (x-y)&\ge (1+p)(|x|+|y|)^{p} |x-y|^2 &&\text{ for all $x,y\in \R^n$.}
\end{align}
\end{lemma}

It is meant, naturally, in \eqref{inq2} and \eqref{inq4} that 
\beqs
|x|^{p}x,|y|^{p} y,(|x|+|y|)^{p} |x-y|^2=0\text{ for  $p\in(-1,0)$ and $x=y=0$.}
\eeqs 

The inequalities \eqref{ele2}--\eqref{inq4} contain explicit constants and many of them are the best ones. For example, inequalities \eqref{inq3m},\eqref{inq2}, and particular cases of \eqref{ele2}, respectively, \eqref{inq1}, \eqref{inq3}, when $p=2$, respectively, $p=1$, $p\ge 1$, have the best constants.
Elementary proofs of these inequalities are provided in Appendix \ref{apA} below.
For inequalities without explicit constants, see \cite{GM1975, Chow1989}.

\medskip
We recall a simple yet very useful inequality from \cite{CHIK1}.

\begin{lemma}[{\cite[from (2.29) to (2.31)]{CHIK1}}]
Let $x,y\in\R^n$ and $\alpha>0$. Then
\beq\label{intineq}
\int_0^1 |tx+(1-t)y|^\alpha\d t\ge \frac{|x-y|^\alpha}{2^{\alpha+1}(\alpha+1)}. 
\eeq
\end{lemma}

\subsection{Functional spaces}
Next, we review the Sobolev spaces and trace theorems.
\emph{Hereafter, the spatial dimension $n\ge 2$ is fixed.}

Let $\Omega$ be an open, bounded subset of $\R^n$ with the boundary $\partial \Omega$ of class $C^1$.

For $1\le s<\infty$, let $L^s(\Omega)$ be the standard Lebesgue space of scalar functions and 
denote $\mathbf L^s(\Omega) =(L^s(\Omega))^n$.   
 The notation $ \norm{\cdot}_{0,s}$ is used to denote both norms $\norm{\cdot}_{L^s(\Omega)}$ and $\norm{\cdot}_{\mathbf L^s(\Omega)}$.

For a nonnegative integer $m$, let $W^{m,p}(\Omega)$ be the standard Sobolev space 
with the norm 
\beqs
\norm{u}_{m,s} =
\Big( \sum_{|\alpha|\le m} \norm{D^\alpha u}^s_{L^s(\Omega)} \Big)^{\frac 1 s}.
\eeqs    

For any normed space $X$, its dual space is denoted by $X'$ and the product between $X'$ and $X$ is denoted by $\inprod{\cdot,\cdot}_{X',X}$, i.e., $\inprod{y,x}_{X',X}=y(x)$ for $y\in X'$ and $x\in X$.

Consider $1< s< \infty$ now. 
The function $\gamma_0:f\in C^\infty(\bar \Omega) \mapsto f\big|_{\partial \Omega}$ can be extended  to 
 a bounded linear mapping $\gamma_{0,s}: W^{1,s}(\Omega) \to L^s(\partial \Omega)$. 
 The function $\gamma_{0,s}(f)$ is called the trace of $f$ on $\partial \Omega$.

Define $X_s=W^{1-1/s,s}(\partial\Omega)$ to be the range of $\gamma_{0,s}$ equipped with the norm
\beqs
\|f\|_{X_s}=\inf\{ \norm{\varphi}_{1,s}: \varphi\in W^{1,s}(\Omega),\gamma_{0,s}(\varphi)=f\}.
\eeqs 
Define the space
\beq \label{Wsdef}
\mathbf W_s(\rm{div},\Omega)=\{ \vv\in \mathbf L^s(\Omega): \diver \vv \in L^s(\Omega)   \}
\eeq 
equipped with the norm 
\beq\label{Wsnorm}
\norm{\vv}_{\mathbf W_s(\rm{div},\Omega)} =\left (\norm{\vv}_{0,s}^s +\norm{\nabla\cdot \vv}_{0,s}^s\right )^{1/s}. 
\eeq 
Then $\mathbf W_s(\rm{div},\Omega)$ is a reflexive Banach space, see Lemma \ref{Wreflex} below.

Let $r>1$ be the H\"older conjugate of $s$, i.e., $1/s+1/r=1$. 
Clearly,  
$$\norm{\vv}_{0,s},\ \norm{\nabla\cdot \vv}_{0,s}\le \norm{\vv}_{\mathbf W_s(\rm{div},\Omega)}.$$ 
Moreover, thanks to \eqref{eleori},
\beq\label{Wadnorm}
\norm{\vv}_{\mathbf W_s(\rm{div},\Omega)} \le \norm{\vv}_{0,s} +\norm{\nabla\cdot \vv}_{0,s}\le 2^{1/r} \norm{\vv}_{\mathbf W_s(\rm{div},\Omega)} . 
\eeq 

Let $\vec{\nu}$ denote the outward normal vector to the boundary $\partial \Omega$.
Then one can extend the normal trace $\gamma_{\rm n}(\vv)=\vv\cdot \vec{\nu}$ for $\vv\in (C^\infty(\overline \Omega ))^n$ to a bounded, linear mapping $\gamma_{{\rm n},s}$ from $\mathbf W_s(\rm{div},\Omega)$ into $X_r'$. In particular, there is $\bar c_1>0$ such that
\beq\label{gamnorm}
\norm{\gamma_{{\rm n},s}(\vv)}_{X_r'}\le \bar c_1\norm{\vv}_{\mathbf W_s(\rm{div},\Omega)} \text{ for all $\vv\in \mathbf W_s(\rm{div},\Omega)$,}
\eeq
and Green's formula 
\beq\label{green}
\int_\Omega \vv \cdot \nabla q \d x +\int_\Omega \nabla \cdot \vv q \d x =  \inprod{\gamma_{{\rm n},s}(\vv),\gamma_{0,r}(q)}_{X_r',X_r}
\eeq
 holds for every $\vv\in \mathbf W_s(\rm{div},\Omega)$ and $q\in W^{1,r}(\Omega)$. 
In fact, the product $\inprod{\gamma_{{\rm n},s}(\vv),\gamma_{0,r}(q)}_{X_r',X_r}$ in \eqref{green}  is the extension of the boundary integral
$$\int_{\partial \Omega} (\vv\cdot\vec{\mathbf \nu})q\, \d\sigma \text{ for }\vv\in (W^{1,s}(\Omega))^n,q\in W^{1,r}(\Omega).$$

Finally, we recall an important norm estimate, see \cite [Inequality (4.2)]{BN1990} or \cite[Lemma A.3]{KS16}.

\begin{lemma}\label{dualnorm}
Let $r,s\in(1,\infty)$ be  H\"older conjugates of each other and let $V=\mathbf W_s({\rm div},\Omega)$. 
Then there exists a constant $C_*>0$ such that, for all $q\in L^r(\Omega)$, it holds
\beq\label{supinfcdn}
\norm{q}_{0,r}\le C_*\sup_{\vv \in V\backslash\{0\}} \frac{\int_\Omega (\diver \vv) q \, \d x}{\norm{\vv}_{ V} }.
\eeq
\end{lemma} 

A direct proof of Lemma \ref{dualnorm} is given in Appendix \ref{apA}.

\section{Fundamental properties (I)}\label{mod1sec}
Let $B:\R^n\times\R^n\times \R^n\to\R^n$ be defined by 
\beq\label{Bnd}
B(u,v,w)= {\rm diag}[u^T  A_1 v,\ldots, u^T  A_n v]w \quad \text{ for }u,v,w\in\R^n,
\eeq 
where each $A_i$, for $i=1,2,\ldots,n$,  is a nonzero $n\times n$ matrix. This form of $B$ is a generalization of the corresponding term in \eqref{BB0}. 

Our focus in this section is the function $G_B: \R^n\to\R^n$ defined by
\beq\label{GBdef}
G_B(u)=\begin{cases}
    0& \text{ for } u=0,\\
   \begin{displaystyle} \frac{1}{|u|}\end{displaystyle} B(u,u,u)& \text{ for } u\in \R^n\setminus\{0\}.
\end{cases}
\eeq
The following are immediate  properties of $B$ and $G_B$.

\begin{lemma}\label{Bprop}
The following statements hold true.
\begin{enumerate}[label=\tnum]
\item $B$ is a trilinear mapping, and hence, it is continuous. 

\item If $A_i$ is symmetric for all $i=1,2,\ldots,n$, then 
\beq 
\label{uvsym} 
B(u,v,w)=B(v,u,w) \text { for any $u,v,w\in\R^n$.}
\eeq 

\item  If there is $C>0$ such that, for all $i=1,2,\ldots,n$,  
\beq\label{Aihyp}
u^T A_i u\ge C|u|^2 \text{ for any } u,v\in\R^n,
\eeq
then one has 
\beq \label{Blow}
B(u,u,u)\cdot u\ge C|u|^4\text{ and }
G_B(u)\cdot u\ge C|u|^3 \text{ for all }u\in \R^n.
\eeq

\item For any $u,v,w\in\R^n$,
\beq\label{BM}
|B(u,v,w)|\le M_*|u|\,|v|\,|w|,
\text{ where }M_*=\left(\sum_{i=1}^n |A_i|^2\right )^{1/2}.
\eeq 
Consequently, 
\beq \label{Gup}
|G_B(u)|\le M_*|u|^2 \text{ for all $u\in\R^n$.}
\eeq 

\item The function $G_B$ is continuous. More explicitly,  one has, for any $u,v\in \R^n$,
\beq\label{co}
\left |G_B(u)-G_B(v)\right |
\le 2M_*(|u| +|v| )|u-v|.
\eeq
\end{enumerate}
\end{lemma}
\begin{proof}
Parts (i)--(iii) are obvious. For part (iv), we have  
\beqs
|B(u,v,w)|\le \left(\sum_{i=1}^n |u^T  A_i v|^2\right)^{1/2}|w|\le \left(\sum_{i=1}^n |A_i|^2\right)^{1/2}|u|\, |v|\, |w|,
\eeqs
which proves \eqref{BM}. It is clear that inequality \eqref{Gup} follows \eqref{BM} with $u=v=w\ne 0$ and $G_B(0)=0$. 

We prove part (v) now. The fact that $G_B$ is continuous comes from part (i) and \eqref{Gup}.
To prove \eqref{co}, we first consider the case when the line segment connecting $u$ and $v$ does not contain the origin. 
For $t\in [0,1]$, let $w(t)=tu+(1-t)v$.
Then $w(t)\ne 0$ for all $t\in [0,1]$.
Define $z=u-v$    and 
$$h(t)=\frac1{|w(t)|}B(w(t),w(t),w(t))\text{ for }t\in[0,1].$$
Applying the Fundamental Theorem of Calculus to the function $h(t)$ on the interval $[0,1]$, we obtain
\beq \label{B1}
\left |G_B(u)-G_B(v)\right |=|h(1)-h(0)|
\le \int_0^1 |h'(t)| \d t.
\eeq 
Calculating $h'(t)$, we have 
\beqs 
 h'(t)=-\frac{1}{|w|^2}\frac{w\cdot z}{|w|}B(w,w,w)+\frac{1}{|w|}(B(z,w,w)+B(w,z,w)+B(w,w,z)).
 \eeqs 
Applying inequality \eqref{BM} yields 
\beq \label{B2}
|h'(t)|\le 4M_* |w||z|. 
\eeq
Note that
\beq  \label{B3}
\int_0^1 |w(t)| \d t\le \int_0^1 t|u| +(1-t)|v| \d t=\frac12(|u| +|v| ).
\eeq 
Combining \eqref{B1}, \eqref{B2} and \eqref{B3} yields \eqref{co}.

Now, consider the case when the line segment connecting $u$ and $v$ contains the origin.
Since the dimension $n$ is greater than one, for $\varep>0$ sufficiently small, we can find $u^\varep,v^\varep\in \R^n\setminus \{0\}$ such that  the line segment connecting $u^\varep$ and $v^\varep$ does not contain the origin, and $u^\varep\to u$, $v^\varep\to v$ as $\varep\to 0$. 
Applying the inequality \eqref{co} to $u^\varep$ and $v^\varep$, and then letting $\varep\to0$, we obtain \eqref{co} for $u$ and $v$.
\end{proof}

Although the coercivity of $G_B$ in \eqref{Blow} is useful in later investigation of the corresponding systems of PDE, it is not enough. The following monotonicities will be needed.

\begin{definition}\label{Mdef}
Let  $F$ be a function from $\R^n$ to $\R^n$.
\begin{enumerate}[label=\tnum]
    \item We say $F$ is monotone if  
    \beqs
 (F(u)-F(v))\cdot (u-v)\ge 0 \text{ for any $u,v\in \R^n$.}
 \eeqs

 \item  
For $\alpha>0$, we say $F$ is (power) $\alpha$-monotone if there exists a constant $C>0$ such that
 \beqs
 (F(u)-F(v))\cdot (u-v)
 \ge C|u-v|^\alpha \text{ for any $u,v\in \R^n$.}
 \eeqs
\end{enumerate}
\end{definition}

Clearly, if $F$ is $\alpha$-monotone, then it is monotone.
Our goal is to establish the monotonicities in Definition \ref{Mdef} for the function $G_B$.

\subsection{The two-dimensional case}\label{ss2d}
We consider the dimension $n=2$ and a special case of $B$ in \eqref{Bnd}.
Specifically, given $a,b,c,d>0$, we set
\beqs
A_1=\begin{pmatrix}
     a&0\\0&b
    \end{pmatrix},\quad 
A_2=\begin{pmatrix}
     c&0\\0&d
    \end{pmatrix},
\eeqs 
and write \eqref{Bnd} explicitly as 
\beqs 
B(u,v,w)=\begin{pmatrix}
     u^T  A_1 v&0\\0&u^T  A_2 v
    \end{pmatrix} w=\begin{pmatrix}
       (au_1v_1+bu_2v_2)w_1\\ (cu_1v_1+du_2v_2)w_2
      \end{pmatrix}
\eeqs
for vectors $u=(u_1,u_2)$, $v=(v_1,v_2)$ and $w=(w_1,w_2)$ in $\R^2$.

Then \eqref{Aihyp} is true with $c_0=\min\{a,b,c,d\}$.
Moreover, $A_1$ and $A_2$ are symmetric and, consequently, we have property \eqref{uvsym}.

For the sake of convenience in the calculations below, we denote $|v|_i=(v^T  A_i v)^{1/2}$ for $i=1,2$.
Then each $|\cdot|_i$ is a norm on $\R^n$.

\begin{example}[Counter example for the monotonicities]\label{nomono}
We give an example to show that the monotonicities in Definition \ref{Mdef} are not automatically satisfied for any positive values of $a,b,c,d$.

\begin{enumerate}[label=\rnum] 
\item \label{nona} 
Let $a=d=1$, $b=c=5$, $u_*=(2,2)$ and $v_*=(3,1)$.
Then $|u_*|=2\sqrt 2$, $|v_*|=\sqrt{10}$,  $B(u_*,u_*,u_*)=48(1,1)$, $B(v_*,v_*,v_*)=(42,46)$ and $u_*-v_*=(-1,1)$.
We have 
\begin{align*}
\left (\frac1{|u_*|}B(u_*,u_*,u_*)-\frac1{|v_*|}B(v_*,v_*,v_*)\right)\cdot (u_*-v_*)
=\frac{-4}{\sqrt{10}}<0.    
\end{align*}
Thus, $G_B$ is not monotone.

\item\label{nonb} Let $A_0$ be any $2\times 2$ matrix.
Set  $u=tu_*$ and $v=t v_*$ for $t>0$. Then
\begin{align*}
\left (A_0(u-v)+\frac1{|u|}B(u,u,u)-\frac1{|v|}B(v,v,v)\right)\cdot (u-v)
=t^2(-1,1)\cdot A_0(-1,1)  - \frac{4t^3}{\sqrt{10}}
\end{align*}
which is  negative for sufficiently large $t$.
Therefore, the function $u\mapsto A_0u+G_B(u)$ is not monotone.
\end{enumerate}
\end{example}

Because of Example \ref{nomono}, it is not guaranteed that the experimental data \eqref{abcd} yields a corresonding monotone function $G_B$. Therefore, different sufficient conditions are needed.

\begin{theorem}\label{Monothm2}
The following statements hold true.
\begin{enumerate}[label=\tnum]
\item If 
\beq\label{simcond1}
\left( b+c - \frac{|w|_1^2+|w|_2^2}{2} \right)^2\le  4ad\text{ for all $w\in\R^2$ with $|w|=1$},
\eeq
then $G_B$ is monotone.

\item If the inequality in \eqref{simcond1} is strict,  
then $G_B$ is $3$-monotone.
\end{enumerate}
\end{theorem}
\begin{proof}
We prove part (i) first. Assume \eqref{simcond1}. Let $u,v\in \R^2$.
By the same perturbation arguments in the last part of the proof of Lemma \ref{Bprop}, 
we can assume, without loss of generality, that the line segment connecting $u$ and $v$ does not contain the origin. 

Let $z=u-v$, and, for $t\in[0,1]$, $w(t)=tu+(1-t)v$ and
\beq \label{Bwz}
h(t)=\frac1{|w(t)|}B(w(t),w(t),w(t))\cdot z.
\eeq 
Note that $w(t)\ne 0$ for all $t\in[0,1]$.
We have 
\beq\label{Iint}
I\eqdef \left(\frac{1}{|u|}B(u,u,u)-\frac{1}{|v|}B(v,v,v)\right)\cdot (u-v)
=h(1)-h(0)=\int_0^1 h'(t)\d t.
\eeq
We calculate the derivative
\begin{align*}
 h'(t)&=-\frac{w\cdot z}{|w|^3}B(w,w,w)\cdot z+\frac{1}{|w|}(B(z,w,w)+B(w,z,w)+B(w,w,z))\cdot z\\
 &=\frac{1}{|w|} (J_1+J_2+J_3),
 \end{align*}
 where
 \begin{align*}
J_1 &=-\frac{w\cdot z}{|w|^2}B(w,w,w)\cdot z, \quad
J_2=(B(w,z,w)+B(z,w,w))\cdot z, \quad
J_3=B(w,w,z)\cdot z.
\end{align*}

More specifically, with $z=(z_1,z_2)$ and $w=(w_1,w_2)$, we have
\begin{align}
J_1 &=-\frac{w\cdot z}{|w|^2}\Big[(w^TA_1 w) w_1z_1+(w^TA_2 w) w_2z_2\Big] \notag\\
&=-\frac1{|w|^2}\Big[|w|_1^2 w_1^2z_1^2 + (|w|_1^2+|w|_2^2)w_1w_2z_1z_2+|w|_2^2 w_2^2z_2^2\Big].\label{J1}
\end{align}

By the symmetry \eqref{uvsym},
\begin{align}
J_2 &=2B(w,z,w)\cdot z=2\Big[(w^TA_1 z) w_1z_1 +(w^TA_2 z) w_2 z_2\Big] \notag \\
&=2\Big[a w_1^2z_1^2 +(b+c)w_1w_2z_1z_2+d w_2^2 z_2^2\Big].\label{J2}
\end{align} 

The last term $J_3$ is 
\beq\label{J3}
J_3 =(w^TA_1 w) z_1^2 +(w^TA_2 w) z_2^2
=|w|_1^2 z_1^2 +|w|_2^2 z_2^2.
\eeq

Combining the terms $J_1$ and $J_3$, we have
\beqs 
 J_1+J_3=\frac1{|w|^2}\Big[-|w|_1^2 w_1^2z_1^2 - (|w|_1^2+|w|_2^2)w_1w_2z_1z_2-|w|_2^2 w_2^2z_2^2
 +(|w|_1^2 z_1^2+|w|_2^2 z_2^2)(w_1^2+w_2^2)\Big]\\
\eeqs 
 which yields
\beq\label{J13}
J_1+J_3
=\frac1{|w|^2}\Big[|w|_1^2 z_1^2w_2^2 +|w|_2^2 z_2^2w_1^2-(|w|_1^2+|w|_2^2)w_1w_2z_1z_2\Big].
\eeq 
It follows \eqref{J13} that
\beqs 
 J_1+J_3\ge -\frac1{|w|^2}(|w|_1^2+|w|_2^2)w_1w_2z_1z_2.
\eeqs 
Combining this with formula \eqref{J2} of $J_2$, we obtain
\beq\label{hquad}
 h'(t)
   \ge \frac{2}{|w|}\left\{ a w_1^2z_1^2 +\left[(b+c) - \frac{|w|_1^2+|w|_2^2 }{2|w|^2} \right]w_1w_2z_1z_2+d w_2^2 z_2^2\right\}.
\eeq

Fix $t\in[0,1]$. Let $\alpha(t)=w_1(t)z_1$ and $\beta(t)=w_2(t)z_2$.
The expression between the parentheses on the right-hand side of \eqref{hquad} is $Q(\alpha(t),\beta(t))$, where $Q$ is the following $t$-dependent  quadratic function 
\beqs
Q (\alpha,\beta)=a \alpha^2 +\left[(b+c) - \frac{|w(t)|_1^2+ |w(t)|_2^2 }{2|w(t)|^2} \right]\alpha\beta+d \beta^2 \text{ for $\alpha,\beta\in\R$.}
\eeqs
By \eqref{simcond1}, we have $Q(\alpha,\beta)\ge 0$ for all $\alpha,\beta\in \R$.
Consequently, $h'(t)\ge 0$ for all $t\in[0,1]$, which gives 
\beq \label{IBu}
\left(\frac{1}{|u|}B(u,u,u)-\frac{1}{|v|}B(v,v,v)\right)\cdot (u-v)\ge 0
\eeq 
 thanks to \eqref{Iint}. 
 Therefore, $G_B$ is monotone.

\medskip
We prove part (ii) now.   Assume \eqref{simcond1} with the strict inequality.
Similar to the proof of part (i), we can assume  that the line segment connecting $u$ and $v$ does not contain the origin.
We refine the estimates in part (i).  Let $\theta,\theta'\in (0,1)$ with $\theta+\theta'=1$.
Elementary manipulations starting with \eqref{J1} and \eqref{J3} give 
\begin{align*}
J_1+J_3&= -\frac1{|w|^2}\Big[(\theta+\theta')|w|_1^2 w_1^2z_1^2 + (|w|_1^2+|w|_2^2)w_1w_2z_1z_2+(\theta+\theta')|w|_2^2 w_2^2z_2^2\Big]+|w|_1^2 z_1^2 +|w|_2^2 z_2^2\\
&=\frac1{|w|^2}\Big[-\theta (|w|_1^2 w_1^2  z_1^2 +|w|_2^2 w_2^2 z_2^2)
- (|w|_1^2+|w|_2^2)w_1w_2z_1z_2\Big]+J_*,
\end{align*}
where
\beq\label{Jztp}
J_*=\left(1-\theta'\frac{w_1^2}{|w|^2}\right)|w|_1^2 z_1^2 
+\left(1- \theta'\frac{w_2^2}{|w|^2}\right)|w|_2^2 z_2^2.
\eeq
Adding this to the formula \eqref{J2} of $J_2$ and  combining the like-terms, we obtain
\beq\label{wh1}
 |w|h'(t)
=J_1+J_2+J_3  = J_*+2J_4,
\eeq 
where
\beqs
J_4
=\left(a-\theta \frac{|w|_1^2}{2|w|^2}\right) w_1^2z_1^2 
+\left[(b+c)- \frac{|w|_1^2+|w|_2^2}{2|w|^2}\right]w_1w_2z_1z_2  +\left(d-\theta\frac{|w|_2^2}{2|w|^2}\right) w_2^2 z_2^2.
\eeqs 

Letting $\bar\alpha=w_1z_1$ and $\bar\beta =w_2z_2$, we rewrite $J_4$ as
\beq \label{J4}
\begin{aligned} 
J_4
&=\left(a-\theta \frac{|w|_1^2}{2|w|^2}\right) \bar\alpha^2 
+\left[(b+c)- \frac{|w|_1^2+|w|_2^2}{2|w|^2}\right]\bar\alpha\bar\beta 
+\left(d-\theta\frac{|w|_2^2}{2|w|^2}\right) \bar\beta^2.
\end{aligned}
\eeq 
The right-hand side of \eqref{J4} is a quadratic function of $\bar\alpha$ and $\bar\beta$.
Note that
\beq\label{norms12}
C_1\le \frac{|w|_1^2}{|w|^2},\frac{|w|_2^2}{|w|^2}\le C_2 \text{ for all $w\in \R^2\setminus\{0\}$.}
\eeq
Thanks to the strict inequalities in \eqref{simcond1},  we can choose $\theta\in(0,1)$ sufficiently small so that 
\beqs
a-C_2\theta>0 \text{ and }
\max_{|w|=1} \left[b+c- \frac{|w|_1^2+ |w|_2^2}{2}\right]^2\le 4\left(a-\frac{C_2\theta }{2}\right) \left(d-\frac{C_2\theta }{2}\right) .
\eeqs
Therefore, for all $w\in \R^2\setminus\{0\}$,  
\beq \label{atheta}
a-\theta \frac{|w|_1^2}{|w|^2} >0
\eeq 
and
\beq\label{bctheta}
\left[b+c- \frac{|w|_1^2+ |w|_2^2}{2|w|^2}\right]^2\le 4\left(a-\theta \frac{|w|_1^2}{2|w|^2}\right) \left(d-\theta\frac{|w|_2^2}{2|w|^2}\right) .
\eeq
Then \eqref{atheta} and \eqref{bctheta} imply that  $J_4\ge 0$ for all $t\in[0,1]$.
Therefore, $ |w|h'(t)\ge J_*$. 

Using the fact $w_1^2,w_2^2\le |w|^2$ in the formula \eqref{Jztp} of $J_*$ and then property \eqref{norms12}, we have
\beq\label{Jest2}
 J_*
 \ge (1-\theta')(|w|_1^2 z_1^2 +|w|_2^2 z_2^2)
 \ge C_1(1-\theta')|w|^2 |z|^2= c_*|w(t)|^2|z|^2,\text{ where } c_*=C_1\theta .
\eeq
It follows that $h'(t)\ge c_* |w(t)||z|^2$.
Combining this estimate of $h'(t)$ with \eqref{Iint} and applying inequality \eqref{intineq} with $\alpha=1$, we have
\beqs
I\ge c_* |z|^2 \int_0^1 |w(t)|\d t \ge \frac{c_*}{8}|z|\cdot |z|^2=\frac{c_*|z|^3}{8}.
\eeqs
Thus, we obtain 
 \beq\label{IB3}
 \left(\frac{1}{|u|}B(u,u,u)-\frac{1}{|v|}B(v,v,v)\right)\cdot (u-v)
 \ge C|u-v|^3
 \eeq
with $C=c_*/8$. Therefore, $G_B$ is $3$-monotone.
\end{proof}

We find a consequence of Theorem \ref{Monothm2} with simpler conditions which can be checked easily. 

\begin{corollary}\label{Cor3}
If 
\beq\label{ok1}
\max\{a+c,b+d\}\le 2(b+c+2\sqrt{ad})
\eeq
and
\beq \label{ok2}
2(b+c-2\sqrt{ad})\le \min\{a+c,b+d\},
\eeq
then $G_B$ is monotone.

If both inequalities \eqref{ok1} and \eqref{ok2} are strict, then $G_B$ is $3$-monotone.
\end{corollary}
\begin{proof}
    Condition \eqref{simcond1} is equivalent to
\beq\label{pos2b}
2(b+c-2\sqrt{ad})\le |w|_1^2+|w|_2^2\le  2(b+c+2\sqrt{ad})\quad \forall |w|=1.
\eeq

Note that
\beq
\min\{a+c,b+d\} |w|^2\le |w|_1^2+|w|_2^2\le \max\{a+c,b+d\}|w|^2.
\eeq
Then \eqref{ok1} are \eqref{ok2} are sufficient conditions for \eqref{pos2b}, and hence $G_B$ is monotone thanks to Theorem \ref{Monothm2}.
The case of strict inequalities in \eqref{ok1} and \eqref{ok2} is similar.
\end{proof}

\begin{example}\label{eg1}
   Define, for $a,b,c,d>0$, the functions
\begin{align*} 
g_1(a,b,c,d)&= 2(b+c+2\sqrt{ad})-\max\{a+c,b+d\},\\
g_2(a,b,c,d)&= \min\{a+c,b+d\}-2(b+c-2\sqrt{ad}).
\end{align*} 

\begin{enumerate}[label=\rnum]
\item For any $a,b>0$, taking $a=d$ and $b=c$, we have
\beq
g_1(a,b,b,a)=3(a+b)>0,\quad g_2(a,b,b,a)=a+b-4(b-a)=5a-3b.
\eeq
If $a=3b/5$, then $G_B$ is monotone. If $a>3b/5$, then $G_B$ is $3$-monotone.

\item In particular, when $a=b=c=d$, the function $G_B$ is $3$-monotone, which is well-known for the generalized Forchheimer equation \eqref{genF}.

\item\label{abstar} Suppose $a_*>3b_*/5>0$. If $(a,b,c,d)$ is sufficiently close to $(a_*,b_*,b_*,a_*)$, then, thanks to the continuity of $g_1$ and $g_2$, one has 
$g_1(a,b,c,d)>0$ and $g_2(a,b,c,d)>0$ which imply $G_B$ is $3$-monotone.

\item If the matrices $A_1$ and $A_2$ are close to $a_* I_2$, for some $a_*>0$, then, thanks to \ref{abstar} with $a_*=b_*$, the function $G_B$ is $3$-monotone.
\end{enumerate}
\end{example}

\begin{example}
With the values $a,b,c,d$ in \eqref{abcd}, we rewrite conditions \eqref{ok1} and \eqref{ok2} as
\beq\label{test1}
b+d\le 2(b+c+2\sqrt{ad})
\text{ and } 
2(b+c-2\sqrt{ad})\le a+c.
\eeq 
The values in \eqref{abcd} satisfy the first inequality in \eqref{test1}, but fail the second one. Thus, the monotonicity of $G_B$ is not yet determined. Therefore, there is a need to derive different  sufficient  conditions for the monotonicities of $G_B$.
\end{example}

Below are other monotonicity conditions which resemble but neither contain nor are contained in those of Theorem \ref{Monothm2}.

\begin{theorem}\label{Monothm1}
The following statements hold true.

\begin{enumerate}[label=\tnum]
    \item\label{M1} If either 
\beq\label{ocond1}
\left[b+c - \frac{(|w|_1+|w|_2)^2 }{2}\right]^2\le 4ad \text{ for all $w\in\R^2$ with $|w|=1$},
\eeq
or
\beq\label{ocond2}
\left[b+c - \frac{(|w|_1-|w|_2)^2 }{2}\right]^2\le 4ad \text{ for all $w\in\R^2$ with $|w|=1$},
\eeq
then $G_B$ is monotone.
 
 \item\label{M2} If either the inequality in \eqref{ocond1} is strict, or the inequality in \eqref{ocond2} is strict, then $G_B$ is $3$-monotone.
\end{enumerate}
\end{theorem}
\begin{proof}
We follow the proof of Theorem \ref{Monothm2} up to \eqref{J13}.

\medskip
\noindent\textit{Part \ref{M1}.} Applying  Cauchy's inequality to the sum $|w|_1^2 z_1^2w_2^2 +|w|_2^2 z_2^2w_1^2$ in \eqref{J13} gives 
\begin{align*}
J_1+J_3
&\ge \frac1{|w|^2}\Big[ \pm 2  |w|_1z_1w_2|w|_2z_2w_1-(|w|_1^2+|w|_2^2)w_1w_2z_1z_2\Big]
= -\frac1{|w|^2}(|w|_1\mp |w|_2)^2w_1w_2z_1z_2.
\end{align*}
Combining with the formula \eqref{J2} of $J_2$, we obtain
\beq\label{hquad2}
 h'(t)
   \ge \frac{2}{|w|}\left\{ a w_1^2z_1^2 +\left[(b+c) - \frac{(|w|_1\mp |w|_2)^2 }{2|w|^2} \right]w_1w_2z_1z_2+d w_2^2 z_2^2\right\}.
\eeq

Fix $t\in[0,1]$. Let $\alpha(t)=w_1(t)z_1$ and $\beta(t)=w_2(t)z_2$.
The expression between the parentheses on the right-hand side of \eqref{hquad2} is $Q_\mp(\alpha(t),\beta(t))$, where $Q_{\mp}$ is the following $t$-dependent  quadratic function 
\beqs
Q_\mp (\alpha,\beta)=a \alpha^2 +\left[(b+c) - \frac{(|w(t)|_1\mp |w(t)|_2)^2 }{2|w(t)|^2} \right]\alpha\beta+d \beta^2 \text{ for $\alpha,\beta\in\R$.}
\eeqs

Suppose \eqref{ocond1} is true.
Then
\beqs
\left[b+c - \frac{(|w|_1+|w|_2)^2 }{2|w|^2}\right]^2\le 4ad\quad\forall w\in \R^2\setminus\{ 0\},
\eeqs
which yields $Q_+(\alpha,\beta)\ge 0$ for all $\alpha,\beta\in\R$.
Consequently, $h'(t)\ge 0$ for all $t\in[0,1]$, which gives $I\ge 0$ thanks to \eqref{Iint} and proves \eqref{IBu}.

In the case \eqref{ocond2}, the proof is similar by using $Q_-$ in place of $Q_+$.

The rest of the proof of part (i) is the same as that of Theorem \ref{Monothm2}(i).

\medskip
\noindent\textit{Part \ref{M2}.}  
We follow the proof of part (ii) of Theorem \ref{Monothm2} up to \eqref{wh1}.
In the remainder of the proof, we use the plus sign for $\pm$ in the case of the strict inequality in \eqref{ocond1}, and use the minus sign  for $\pm$ in the case  of the strict inequality in \eqref{ocond2}

Adding and subtracting $\pm 2\frac{|w|_1|w|_2}{|w|^2}w_1w_2z_1z_2$ to the last sum of \eqref{wh1}, we obtain
\beq\label{wh2}
\begin{aligned}
 |w|h'(t)
 &  = J_* \pm 2\frac{|w|_1|w|_2}{|w|^2}w_1w_2z_1z_2 +2\widetilde J_4,
\end{aligned}
\eeq 
where
\begin{align*}
\widetilde J_4
&=J_4-\left (\pm\frac{|w|_1|w|_2}{|w|^2}w_1w_2z_1z_2\right )\\
&=\left(a-\theta \frac{|w|_1^2}{2|w|^2}\right) w_1^2z_1^2 
+\left[(b+c)- \frac{(|w|_1^2+|w|_2^2\pm 2|w|_1|w|_2)}{2|w|^2}\right]w_1w_2z_1z_2  +\left(d-\theta\frac{|w|_2^2}{2|w|^2}\right) w_2^2 z_2^2.
\end{align*} 

Letting $\bar\alpha=w_1z_1$ and $\bar\beta =w_2z_2$, we rewrite $\widetilde J_4$ as
\beq \label{J4til}
\begin{aligned} 
\widetilde J_4
&=\left(a-\theta \frac{|w|_1^2}{2|w|^2}\right) \bar\alpha^2 
+\left[(b+c)- \frac{(|w|_1\pm |w|_2)^2}{2|w|^2}\right]\bar\alpha\bar\beta 
+\left(d-\theta\frac{|w|_2^2}{2|w|^2}\right) \bar\beta^2.
\end{aligned}
\eeq 
The right-hand side of \eqref{J4til} is a quadratic function of $\bar\alpha$ and $\bar\beta$.
Then there is $\theta\in(0,1)$ sufficiently small so that one has, for all $w\in \R^2\setminus\{0\}$,  
\eqref{atheta} holds true and, similar to \eqref{bctheta},
\beq\label{plustheta}
\left[b+c- \frac{(|w|_1\pm |w|_2)^2}{2|w|^2}\right]^2\le 4\left(a-\theta \frac{|w|_1^2}{2|w|^2}\right) \left(d-\theta\frac{|w|_2^2}{2|w|^2}\right) .
\eeq

Note that $\theta'=1-\theta$ is  a fixed number in $(0,1)$ now.
Then \eqref{atheta} and \eqref{plustheta} imply that  $\widetilde J_4\ge 0$ for all $t\in[0,1]$.
Combining this fact with formulas \eqref{wh2} and \eqref{Jztp} gives
\begin{align*}
 |w|h'(t)
 &\ge J_* \pm 2\frac{w_1w_2}{|w|^2}(|w|_1z_1)(|w|_2z_2)\\
 &= \left(1-\theta'\frac{w_1^2}{|w|^2}\right)|w|_1^2 z_1^2 
+\left(1- \theta'\frac{w_2^2}{|w|^2}\right)|w|_2^2 z_2^2
\pm 2\frac{w_1w_2}{|w|^2}(|w|_1z_1)(|w|_2z_2).
\end{align*}
This implies 
\beq\label{whQ}
 h'(t) \ge \frac1{|w|}\widetilde Q(|w|_1 z_1,|w|_2z_2),
 \eeq 
where $\widetilde Q$ is now a new quadratic function defined for each $t$ by
\beqs 
\widetilde Q(\alpha,\beta)=
\left(1-\theta'\frac{w_1^2(t)}{|w(t)|^2}\right)\alpha^2 
+\left(1- \theta'\frac{w_2^2(t)}{|w(t)|^2}\right)\beta^2
\pm 2\frac{w_1(t)w_2(t)}{|w(t)|^2}\alpha\beta \text{ for }\alpha,\beta\in\R. 
\eeqs 

We claim that 
\beq\label{ww0}
w_1^2 w_2^2<  (1-\theta'w_1^2)(1- \theta'w_2^2) \text{ for all $w\in\R^2$ with $|w|=1$.}
\eeq
Indeed, \eqref{ww0} is equivalent to
\begin{align*}
0< -w_1^2 w_2^2+1-\theta'(w_1^2+w_2^2) +  {\theta'}^2 w_1^2w_2^2
=1-\theta'-( 1 -{\theta'}^2) w_1^2w_2^2,
\end{align*}
which, in turn, is equivalent to 
\beq\label{ww} 
w_1^2w_2^2<\frac{1}{1+\theta'}.
\eeq 
Because 
\[
w_1^2w_2^2 \le \left(\frac{w_1^2+w_2^2}{2}\right)^2 =\frac 1 4,\text{ while }
\frac{1}{1+\theta'}\ge \frac12,
\]
one asserts that \eqref{ww} is true, and, hence, so is the claim \eqref{ww0}. 

By \eqref{ww0} and the compactness of the unit sphere $\{w\in\R^2:|w|=1$\}, we have
\beq\label{ww2}
1>\max_{|w|=1}\frac{w_1^2 w_2^2}{(1-\theta'w_1^2)(1- \theta'w_2^2)}=(1-\varep)^2,
\eeq
for some number $\varep\in(0,1)$.
We split $\widetilde Q$ as
\beq\label{QQ}
\widetilde Q(\alpha,\beta)=
\varepsilon \left(1-\theta'\frac{w_1^2(t)}{|w(t)|^2}\right)\alpha^2 
+\varepsilon \left(1- \theta'\frac{w_2^2(t)}{|w(t)|^2}\right)\beta^2
+\widehat Q(\alpha,\beta),
\eeq 
where, for each $t$, 
\beqs
\widehat Q(\alpha,\beta)=(1-\varep)\left(1-\theta'\frac{w_1^2(t)}{|w(t)|^2}\right)\alpha^2
+(1-\varep)\left(1- \theta'\frac{w_2^2(t)}{|w(t)|^2}\right)\beta^2
\pm 2\frac{w_1(t)w_2(t)}{|w(t)|^2}\alpha\beta .
\eeqs
Because
$$(1-\varep)\left(1-\theta'\frac{w_1^2}{|w|^2}\right)\ge (1-\varep)(1-\theta')>0,$$
and, thanks to \eqref{ww2},
\beqs
\left(\frac{w_1}{|w|}\frac{ w_2}{|w|}\right)^2
\le (1-\varep)^2 \left (1-\theta'\frac{w_1^2}{|w|^2}\right) 
\left (1- \theta'\frac{w_2^2}{|w|^2}\right)\quad\forall w\in\R^2,w\ne 0,
\eeqs
we have $\widehat Q\ge 0$. Together with \eqref{whQ} and \eqref{QQ}, it implies
$h'(t) \ge   \varep J_*/|w|$, where $J_*$ is defined by \eqref{Jztp}.
Thanks to the lower bound of $J_*$ in \eqref{Jest2}, we derive
$h'(t)\ge \varep c_*|w(t)||z|^2$.
By using \eqref{Iint} and applying inequality \eqref{intineq} with $\alpha=1$, we 
obtain \eqref{IB3} for $C=\varep c_*/8$ and complete proof of part (ii).
\end{proof}

Note that condition \eqref{ocond1} is equivalent to
\beq\label{pos1}
2(b+c-2\sqrt{ad})\le (|w|_1+ |w|_2)^2\le 2(b+c+2\sqrt{ad}) \text{ for all $w\in\R^2$ with $|w|=1$,}
\eeq
while condition \eqref{ocond2} is equivalent to
\beq\label{pos2}
2(b+c-2\sqrt{ad})\le (|w|_1- |w|_2)^2\le 2(b+c+2\sqrt{ad})\text{ for all $w\in\R^2$ with $|w|=1$,}
\eeq

Based on Theorem \ref{Monothm1}, more specific conditions for the monotonicities can then be derived in the following.
 
\begin{corollary}\label{Cor1}
If
\beq\label{good1}
\max\{a+c,b+d\} +2\sqrt{\max\{a,b\}}\sqrt{\max\{c,d\}}\le  2(b+c+2\sqrt{ad})
\eeq
and
\beq\label{good2}
2(b+c-2\sqrt{ad}) \le \min\{a+c,b+d\} +2\sqrt{\min\{a,b\}}\sqrt{\min\{c,d\}},
\eeq
then $G_B$ is monotone.

If both inequalities \eqref{good1} and \eqref{good2} are strict, then $G_B$ is $3$-monotone.
\end{corollary}
\begin{proof}
We use condition \eqref{pos1} for Theorem \ref{Monothm1}.
Note, for any $w\in\R^2$,  that 
\beqs
(|w|_1+|w|_2)^2=|w|_1^2+|w|_2^2+2|w|_1|w|_2\le \max\{a+c,b+d\} |w|^2+2\sqrt{\max\{a,b\}}\sqrt{\max\{c,d\}}|w|^2,
\eeqs
and
\beqs
(|w|_1+|w|_2)^2\ge \min\{a+c,b+d\}|w|^2+2\sqrt{\min\{a,b\}}\sqrt{\min\{c,d\}}|w|^2.
\eeqs
Hence, \eqref{good1} is a sufficient condition for the second inequality in \eqref{pos1}, and \eqref{good2} is a sufficient condition for the first inequality in \eqref{pos1}.
Thus, the first statement follows part (i) of Theorem \ref{Monothm1}.

The arguments for the case \eqref{good1} and \eqref{good2} being strict inequalities  are similar by using part (ii) of Theorem \ref{Monothm1}.
\end{proof}

We are ready to check our mathematical criteria with the real data in \cite{BarakBear81}.

\begin{example}\label{eg2}
With the experimental values in \eqref{abcd} we have $a<c<b<d$. Then the strict inequalities in \eqref{good1} and \eqref{good2} become
\beq\label{match1}
b +d+2\sqrt{bd}< 2(b+c+2\sqrt{ad})
\eeq
and
\beq\label{match2}
2(b+c-2\sqrt{ad}) < ( a+c +2\sqrt{ac}).
\eeq
Simple calculations show that the values in \eqref{abcd} satisfy both \eqref{match1} and \eqref{match2}. Therefore, $G_B$ is $3$-monotone. Moreover, thanks to the same continuity argument in part (b) of Example \ref{eg1},  $G_B$ is still $3$-monotone even when the measurements have small errors. This shows that our mathematical criteria work for real data.
\end{example}

Using condition \eqref{pos2} instead of \eqref{pos1}, we obtain a counterpart of Corollary \ref{Cor1}.

\begin{corollary}\label{Cor2}
The following statements hold true.

\begin{enumerate}[label=\tnum]
\item\label{xsim1} If 
\beq \label{sn1}
\frac{\max\{a,d\}}2 -2\sqrt{ad} \le b+c\le 2\sqrt{ad},
\eeq 
then $G_B$ is monotone.

\noindent If 
\beq \label{sn2}
\frac{\max\{a,d\}}2 -2\sqrt{ad} < b+c < 2\sqrt{ad},
\eeq 
 then  $G_B$ is $3$-monotone.

\item Let $M=\max\{a,d\}$ and $m=\min\{a,d\}$.

\begin{enumerate}[label=\rnum]
\item If $M/m\le 16$, then \eqref{sn1}, respectively, \eqref{sn2},  is equivalent to
\beqs
b+c\le 2\sqrt{ad}, \text{ respectively, }
b+c< 2\sqrt{ad}.
\eeqs

\item If $16<M/m< 64$, then \eqref{sn1} and 
\eqref{sn2} stay the same.

\item If $M/m=64$,  then \eqref{sn1} is equivalent to
\beqs
b+c= 2\sqrt{ad},
\eeqs
while 
\eqref{sn2}  is impossible.

\item If $M/m>64$, then \eqref{sn1}  is impossible.
\end{enumerate}
\end{enumerate}
\end{corollary}
\begin{proof}
Observe, for all $w\in\R^2$, that 
\beq\label{wmw}
(|w|_1- |w|_2)^2\le \max\{|w|_1^2,|w|_2^2\}=\max\{a,b,c,d\}|w|^2.
\eeq 

\medskip
Part (i). Assume \eqref{sn1}. By the virtue of part (ii) of Theorem \ref{Monothm1},
we can use  \eqref{pos2} as a sufficient condition for $G_B$ to be monotone. 
Thanks to \eqref{wmw}, one sufficient condition for  \eqref{pos2} is
\beq \label{pos3}
b+c\le 2\sqrt{ad},  \quad \max\{a,b,c,d\} \le 2(b+c+2\sqrt{ad}).
\eeq 
By dropping $b$ and $c$ on the left-hand side of the last inequality, we obtain the sufficient condition \eqref{sn1}.

Now, assume \eqref{sn2}. Then the second inequality in \eqref{sn1} implies the first inequality in \eqref{pos2} is strict. Also, the first inequality in \eqref{sn1} implies the second inequality in \eqref{pos3} is strict. This and \eqref{wmw} imply the second inequality in \eqref{pos2} is also strict.
Thus, the inequality in \eqref{ocond2} is strict. Therefore, according to part (ii) of Theorem \ref{Monothm1},  we have $G_B$ is $3$-monotone.

\medskip
Part (ii). We rewrite the first group in \eqref{sn1} and \eqref{sn2} as
\beqs 
S_1\eqdef \frac{\max\{a,d\}}2 -2\sqrt{ad}=\frac M 2 -2\sqrt{Mm}=\frac12\sqrt{Mm}(\sqrt{M/m}-4).
\eeqs

If $M/m\le 16$, then $S_1\le 0$. Therefore, the statements in part (a) are correct.

Comparing the last group with the first group in \eqref{sn1} and \eqref{sn2}, we consider
\beq
S_2\eqdef 2\sqrt{ad}-\left(\frac{\max\{a,d\}}2 -2\sqrt{ad}\right)=4\sqrt{Mm}-\frac M 2 
=\frac12 \sqrt{Mm} (8-\sqrt{M/m}).
\eeq

If $16<M/m<64$, then $S_1>0$ and $S_2>0$, which yield no changes to \eqref{sn1} and \eqref{sn2}.

When $M/m=64$, one has $S_1>0$ and $S_2=0$. 
Therefore, the statements in part (c) are correct.

If $M/m>64$, then $S_2<0$, and, hence, inequality \eqref{sn1}  is impossible.
\end{proof}

\begin{example}\label{eg3}
If $d=36a$ and $6a <b+c<12 a$, then, by using \eqref{sn2}, we have $G_B$ is $3$-monotone.
\end{example}

\subsection{The case of general dimensions}\label{ssnd}

 We return to the general spatial dimension $n\ge 2$ with the trilinear mapping $B$ defined by \eqref{Bnd}.
 
\begin{theorem}\label{Monothm3}
The following statements hold true.

\begin{enumerate}[label=\tnum]
\item If there exists a number $\lambda>0$ such that 
\beq\label{Ac}
\|A_i-\lambda I_n\|_{\rm op}\le \frac{\lambda}4 \quad\text{for all }i=1,2,\ldots,n,
\eeq
then $G_B$ is monotone.

\item If the inequality in \eqref{Ac} is strict for all $i=1,\ldots,n$, then  $G_B$ is $3$-monotone.
\end{enumerate}
\end{theorem}
\begin{proof}
Part (i). Assume \eqref{Ac}. Denote $A_*=\lambda I_n$ and, for $i=1,2,\ldots,n$,  $M_i=A_i-\lambda I_n$.

Let $u,v\in\R^n$. We can assume, without  loss of generality, that the line segment connecting $u$ and $v$ does not contain the origin. 
Let $z$, $w(t)$ and $h(t)$ be as in \eqref{Bwz}, and $I$ be as in \eqref{Iint}.
We calculate
\begin{align*}
 h'(t)&=-\frac{w\cdot z}{|w|^3}B(w,w,w)\cdot z+\frac{1}{|w|}(B(z,w,w)+B(w,z,w)+B(w,w,z))\cdot z\\
 &=-\frac{w\cdot z}{|w|^3}\sum_{i=1}^n (w^T A_i w) w_iz_i
 +\frac{1}{|w|}\sum_{i=1}^n (z^T A_i w+w^T A_i z) w_iz_i
 +\frac{1}{|w|}\sum_{i=1}^n (w^T A_i w) z_i^2.
\end{align*}

For each $i$, we write $A_i=A_*+M_i$ and split the sums in $h'(t)$ accordingly to yield   $h'(t)=H_1+H_2$, where
\begin{align*}
H_1 &=-\frac{w\cdot z}{|w|^3}\sum_{i=1}^n (w^T A_* w) w_iz_i
+\frac{1}{|w|}\sum_{i=1}^n (z^T A_* w+w^T A_* z) w_iz_i
+\frac{1}{|w|}\sum_{i=1}^n (w^T A_* w) z_i^2\\
 H_2 &= -\frac{w\cdot z}{|w|^3}\sum_{i=1}^n (w^T M_i w) w_iz_i
 +\frac{1}{|w|}\sum_{i=1}^n (z^T M_i w+w^T M_i z) w_iz_i
 +\frac{1}{|w|}\sum_{i=1}^n (w^T M_i w) z_i^2.
\end{align*}

Regarding  $H_1$, one has 
\beqs 
    H_1=-\frac{\lambda|w|^2(w\cdot z)^2}{|w|^3}+\frac{2\lambda(w\cdot z)^2}{|w|}+\frac{1}{|w|}\lambda|w|^2 |z|^2
    =\frac{\lambda(w\cdot z)^2}{|w|}+\lambda|w| |z|^2\ge \lambda|w||z|^2.
\eeqs 
Set $\varep_0=\max\{ \|M_i\|_{\rm op}:1\le i\le n\}$.
Applying the Cauchy--Schwarz inequality and the first inequality in \eqref{op1} multiple times, we estimate $H_2$ by 
\begin{align*}
    |H_2|
    &\le \frac{|w|\cdot |z|}{|w|^3}\sum_{i=1}^n \varep_0 |w|^2 |w_i| |z_i|
    +\frac{1}{|w|}\sum_{i=1}^n 2\varep_0 |w||z| |w_i||z_i|
    +\frac{1}{|w|}\sum_{i=1}^n \varep_0 |w|^2 z_i^2\\
    &= \varep_0\left ( 3|z|\sum_{i=1}^n  |w_i| |z_i| + |w||z|^2\right )
    \le \varep_0\left ( 3|z|\cdot |w||z|+ |w||z|^2\right )
    = 4\varep_0 |w||z|^2.
\end{align*}
Combining the above estimates of $H_1$ and $H_2$ gives
\begin{align*}
 h'(t)
  &\ge (\lambda-4\varep_0) |w||z|^2.
 \end{align*}
Thanks to condition \eqref{Ac}, we have $\varep_0\le \lambda/4$, then $h'(t)\ge 0$ for all $t\in[0,1]$. By this fact and \eqref{Iint}, we obtain \eqref{IBu}. Thus, $G_B$ is monotone.

Part (ii). With strict inequalities in \eqref{Ac}, we have $\lambda-4\varep_0=c>0$. Thus,
$ h'(t) \ge c |w(t)||z|^2$ for all $t\in[0,1]$.
Integrating this inequality from $0$ to $1$ and  then applying inequality \eqref{intineq}, we obtain \eqref{IB3} with $C=c/8$.
Therefore, $G_B$ is $3$-monotone.
\end{proof}

\begin{remark}
Compared with Theorem \ref{Monothm1} in the two-dimension case, Theorem \ref{Monothm3} does not require each matrix $A_i$ to be diagonal. However, it must be a perturbation of the same multiple of the identity matrix for all $i$.
\end{remark}

The next theorem contains very specific sufficient conditions for $G_B$'s monotonicities via justifying \eqref{Ac}. Below, ${\rm Tr}(\cdot)$ denotes the trace of square matrices.
 
\begin{theorem}\label{Monothm4}
If
\beq\label{trA}
\frac{\sum_{i=1}^n {\rm Tr}(A_i)}{\sqrt{\sum_{i=1}^n |A_i|^2}}\ge \sqrt{n - \frac1{16}},
\eeq
then $G_B$ is monotone.

If inequality \eqref{trA} is strict, then $G_B$ is $3$-monotone.
\end{theorem}
\begin{proof}
Assume \eqref{trA}. 
Thanks to Theorem \ref{Monothm3}(i) and the relation \eqref{nnorms}, a sufficient condition for  $G_B$ being monotone is that there is $\lambda>0$ such that 
\beqs
\sum_{i=1}^n |A_i-\lambda I_n|^2\le \lambda^2/16.
\eeqs 
This can be rewritten  as a quadratic inequality in $\lambda$ as 
\beq\label{lameq}
\left(n - \frac1{16}\right)\lambda^2-2\lambda \sum_{i=1}^n {\rm Tr}(A_i)  +\sum_{i=1}^n |A_i|^2\le 0.
\eeq 
By \eqref{trA}, we have 
\beq
\sum_{i=1}^n {\rm Tr}(A_i) >0\text{ and } \left(\sum_{i=1}^n {\rm Tr}(A_i)\right)^2- (n - 1/16)\sum_{i=1}^n |A_i|^2\ge 0.
\eeq
Then the quadratic inequality \eqref{lameq} has a positive solution $\lambda$, which implies  $G_B$ is monotone.

\medskip
If inequality \eqref{trA} is strict, then there is a positive solution $\lambda$ of the corresponding strict inequality in \eqref{lameq}. 
By \eqref{nnorms} and Theorem \ref{Monothm3}(ii), $G_B$ is $3$-monotone.
\end{proof}

\section{Fundamental properties (II)}\label{mod2sec}

Let $A$ and $K$ be $n\times n$ matrices, and $\alpha$ be a positive number.
The functions of our interest are of the general form
\beq \label{AKdef}
F_{A,K,\alpha}(u)=|Ku|^\alpha Au \text{ for $u\in \R^n$.}
\eeq 

In particular, when $K=I_n$ in \eqref{AKdef}, we have the function 
\beq \label{FAdef}
F_{A,\alpha}(u)=|u|^\alpha Au\text{ for  $u\in\R^n$.} 
\eeq 

We note that when $A={\rm diag}[a_1,a_2,\ldots,a_n]$ and $\alpha=1$,  
$$F_{A,1}(u)=|u|^{-1}B(u,u,u)=G_B(u) \text{ for $u\ne 0$, }$$
with $B$ defined as in \eqref{Bnd} for
$A_i=a_i I_n$. Thus,  properties of $F_{A,1}$ in this case can be derived from those in Section \ref{mod1sec}.
Nonetheless, we will focus on the general matrix $A$ and power $\alpha$ in this section.

When $A$ is symmetric, positive definite, and $K=A^{1/2}$ in \eqref{AKdef}, we have the functions
\beq\label{tildeFA}
\widetilde F_{A,\alpha}(u)=|A^{1/2}u|^\alpha Au\text{ for $u\in\R^n$.}
\eeq

For our convenience, we also define
\beq
F_{A,0}(u)=\widetilde F_{A,0}(u)=Au\text{ for $u\in\R^n$.}
\eeq

The following are natural necessary conditions for the monotonicities for $F_{A,K,\alpha}$.

\begin{proposition}\label{MonoNec}
  Let $A,K$ be $n\times n$ matrices and $\alpha>0$. 
  \begin{enumerate}[label=\tnum]
     \item   If $K$ is invertible and $F_{A,K,\alpha}$ is monotone, then 
   \beq\label{FAnec1}
   u^TA u\ge 0\text{ for all $u\in\R^n$.}\eeq 

\item    If $F_{A,K,\alpha}$ is $\beta$-monotone for some $\beta>0$, then $A$ and $K$ are invertible, $\beta=\alpha+2$ and there is $C>0$ such that 
     \beq\label{FAnec2}
   u^TA u\ge C|u|^2\text{ for all $u\in\R^n$.}\eeq 
 \end{enumerate}
\end{proposition}
\begin{proof}
Part (i). Assume $K$ is invertible and  $F_{A,K,\alpha}$ is monotone. Then
 \beq\label{moc1}
 \left(|Ku|^\alpha Au-|Kv|^\alpha Av\right)\cdot (u-v)
 \ge 0
\text{  for any $u,v\in \R^n$.}
 \eeq
By taking $v=0$ in \eqref{moc1}, we obtain \eqref{FAnec1}.

Part (ii). Assume $F_{A,K,\alpha}$ is $\beta$-monotone. Then there is $C_0>0$ such that 
 \beqs
 \left(|Ku|^\alpha Au-|Kv|^\alpha Av\right)\cdot (u-v)
 \ge C_0|u-v|^\beta
\text{  for any $u,v\in \R^n$.}
 \eeqs
Taking $v=0$ again yields 
 \beq\label{moc2}
 |Ku|^\alpha Au\cdot u
 \ge C_0|u|^\beta
\text{  for any $u\in \R^n$.}
 \eeq
This implies $A$ are $K$ are invertible.
It also follows \eqref{moc2} that
\beq\label{moc3}
|K|^\alpha |A| |u|^{\alpha+2}
 \ge C_0|u|^\beta
\text{  for any $u\in \R^n$.}
 \eeq
Taking $u=tw$ for a fixed $w\ne 0$ in \eqref{moc3}, and letting $t\to\infty$ and then $t\to 0^+$ yield $\beta=\alpha+2$. 
As a consequence of this fact and \eqref{moc2},
 \beqs
 |K|^\alpha |u|^\alpha Au\cdot u
 \ge C_0|u|^{\alpha+2}
\text{  for any $u\in \R^n$,}
 \eeqs
which implies \eqref{FAnec2} with $C=C_0/|K|^\alpha$.
\end{proof}

\begin{example} \label{nomoF}
In $\R^2$, let 
\beqs 
A=\begin{pmatrix}
    5 & 1/5\\
    1/5& 1/100
\end{pmatrix},\quad u=(3,5), \text{  and }v=(4,1).
\eeqs
Then $A$ is symmetric and positive definite.
However, $(|u|A u - |v|Av)\cdot (u-v) <0$, which implies $F_{A,1}$ is not monotone.
\end{example}

Because of Example \ref{nomoF}, we need to investigate the monotonicities in different situations.

\begin{theorem}\label{Monothm7}
Suppose there is a number $\kappa_0>0$ such that 
\beq\label{Avv}
u^T  Au\ge \kappa_0^2 |u|^2 \text{ for all $u\in \R^n$.}
\eeq

\begin{enumerate}[label=\tnum]
\item\label{FFm1} Then the function $F_{A,0}=\widetilde F_{A,0}$ is $2$-monotone.

\item\label{FFm2} If $A$ is symmetric and $\alpha>0$, then the function  $\widetilde F_{A,\alpha}$ is $(\alpha+2)$-monotone.

\item\label{FFm3} If $\alpha=\kappa_0^2/\|A\|_{\rm op}$, then the function $F_{A,\alpha}$ is monotone.

\noindent 
If $\alpha< \kappa_0^2/\|A\|_{\rm op}$, then the function $F_{A,\alpha}$ is $(\alpha+2)$-monotone.
\end{enumerate}
\end{theorem}
\begin{proof}
Part \ref{FFm1}. Because $F_{A,0}=\widetilde F_{A,0}$ is linear, the assumption \eqref{Avv} implies that it is $2$-monotone.

\medskip
Part \ref{FFm2}. Denote $F=\widetilde F_{A,\alpha}$. In this case, $A^{1/2}$ is symmetric and we note, for any $u\in\R^n$, that
\beqs
|A^{1/2}u| = (A^{1/2}u\cdot A^{1/2}u)^{1/2} = (u\cdot A^{1/2} A^{1/2}u)^{1/2} = (u\cdot Au)^{1/2} \ge \kappa_0 |u|,
\eeqs
hence,
\beq\label{Fuu}
F(u)\cdot u= |A^{1/2}u|^{\alpha+2} \ge \kappa_0^{\alpha+2} |u|^{\alpha+2}.
\eeq

We will prove that, for any $u,v\in\R^n$,
\beq\label{Ahauv}
I\eqdef (F(u)-F(v))\cdot (u-v)\ge \frac{\kappa_0^{\alpha+2}}{2^{\alpha+1}(\alpha+1)}  |u-v|^{\alpha+2}.
\eeq

Thanks to inequality \eqref{Fuu}, one obtains \eqref{Ahauv} when  $u=0$ or $v=0$.

Consider $u,v\in \R^n\setminus\{0\}$ now. It suffices to establish \eqref{Ahauv} when the line segment connecting $u$ and $v$ does not contain the origin. Denote $z=u-v$, $w(t)=tu+(1-t)v$ for $t\in [0,1]$, and 
\beqs
h(t)=F(w(t))\cdot z=|A^{1/2}w(t)|^\alpha (A^{1/2}w(t)\cdot A^{1/2}z).
\eeqs
Then  $w(t)\ne 0$ for all $t\in[0,1]$, and 
\begin{align*}
    h'(t)
    &=|A^{1/2}w(t)|^\alpha |A^{1/2}z|^2 +\alpha |A^{1/2}w(t)|^{\alpha-2} (A^{1/2}w(t)\cdot A^{1/2}z)^2\\
    &\ge |A^{1/2}w(t)|^\alpha |A^{1/2}z|^2\ge \kappa_0^{\alpha+2}|w(t)|^\alpha |z|^2.
\end{align*}
By The Fundamental Theorem of Calculus and inequality \eqref{intineq}, we obtain
\beqs
I=\int_0^1 h'(t)\d t\ge \frac{\kappa_0^{\alpha+2}}{2^{\alpha+1}(\alpha+1)} |z|^{\alpha+2},
\eeqs
which proves \eqref{Ahauv}.
Therefore, $F$ is $(\alpha+2)$-monotone. 

\medskip
Part \ref{FFm3}. Denote $F=F_{A,\alpha}$. Note, for all $u\in \R^n$, that
$F(u)\cdot u\ge \kappa_0^2 |u|^{\alpha+2}$.

Similar to part \ref{FFm2}, it suffices to consider $u,v\in\R^n\setminus\{0\}$ and the line segment connecting $u$ and $v$ does not contain the origin. With the same $I$, $z$ and $w(t)$ as in part \ref{FFm2}, we define, for $t\in[0,1]$, 
$h(t)=|w(t)|^\alpha Aw(t)\cdot z$.
Then
\begin{align*}
    h'(t)&=|w(t)|^\alpha (Az)\cdot z +\alpha |w(t)|^{\alpha-2} (w(t)\cdot z) Aw(t)\cdot z\\
    &\ge \kappa_0^2 |w(t)|^\alpha |z|^2 -\alpha |w(t)|^{\alpha-2} \cdot |w(t)||z| \cdot \|A\|_{\rm op} |w(t)||z|\\
    &= (\kappa_0^2-\alpha \|A\|_{\rm op}) |w(t)|^\alpha |z|^2.
\end{align*}

If $\alpha=\kappa_0^2/\|A\|_{\rm op}$, then $h'(t)\ge 0$, which implies $I\ge 0$. 
Therefore, $F_{A,\alpha}$ is monotone. 

If $\alpha<\kappa_0^2/\|A\|_{\rm op}$, Then $c\eqdef \kappa_0^2-\alpha \|A\|_{\rm op}$ is positive.
By \eqref{intineq}, we have
\beq
I=\int_0^1 h'(t)\d t\ge \frac{c|z|^{\alpha}}{2^{\alpha+1}(\alpha+1)}|z|^2.
\eeq
Thus, $F_{A,\alpha}$ is $(\alpha+2)$-monotone. 
\end{proof}

Part \ref{FFm2} of Theorem \ref{Monothm7} means that the flow is close to (anisotropic) Darcy's flows ($\alpha=0$).
Also, if we assume additionally that $A$ is symmetric, then we can take $\kappa_0^2$ to be the smallest eigenvalue $\lambda_{\min}$ of $A$, while $\|A\|_{\rm op}$ is then largest eigenvalue $\lambda_{\max}$ of $A$. Thus, the quotient $\kappa_0^2/\|A\|_{\rm op}$ is simply $\lambda_{\min}/\lambda_{\max}$.

For large $\alpha$, the following is a counterpart of Theorem \ref{Monothm3}.

\begin{theorem}\label{Monothm5}
Let $\alpha>0$.
   \begin{enumerate}[label=\tnum]
       \item If there exists $\lambda>0$ such that
    \beq\label{bbs}
    \|A-\lambda I_n\|_{\rm op}\le \frac{\lambda}{1+\alpha},
    \eeq
then $F_{A,\alpha}$ is monotone.

\item If inequality \eqref{bbs} is a strict inequality, then $F_{A,\alpha}$ is $(\alpha+2)$-monotone.
   \end{enumerate} 
\end{theorem}
\begin{proof}
Part (i).  
Given $u,v\in\R^n$. Define $z=u-v$, and 
$$w(t)=tu-(1-t)v,\quad h(t)=|w(t)|^\alpha Aw(t)\cdot z \text{ for $t\in[0,1]$.}$$
Again, without loss of generality, we assume $w(t)\ne 0$ for all $t\in[0,1]$. 
Using the Fundamental Theorem of Calculus, we have
$$I\eqdef (|u|^\alpha A u-|v|^\alpha A(v))\cdot (u-v)=h(1)-h(0)=\int_0^1 h'(t)\d t.$$

We calculate 
\beq
h'(t)=|w|^\alpha Az\cdot z +\alpha |w|^{\alpha-2} (w\cdot z)Aw\cdot z.
\eeq

Denote $M=A-\lambda I_n$. Then, thanks to \eqref{bbs},  $\|M\|_{\rm op}\le \lambda/(1+\alpha)$. 

With $A=\lambda I_n+M$, we have
\beq\label{hp}
\begin{aligned}
h'(t)
&=\lambda|w|^\alpha|z|^2+|w|^\alpha M z\cdot z
+\lambda|w|^{\alpha-2}(w\cdot z)^2 +\alpha |w|^{\alpha-2}(w\cdot z)M w\cdot z \\
&\ge(\lambda-(1+\alpha)\|M\|_{\rm op}  )|w|^\alpha |z|^2.
\end{aligned}
\eeq 
Thus, $h'(t)\ge 0$ which implies $I\ge 0$, that is, $F_{A,\alpha}$ is monotone.

Part (ii). In this case, we have  $c\eqdef \lambda-(1+\alpha)\|M\|_{\rm op}>0$. Combining this with \eqref{hp} and then applying inequality \eqref{intineq} give 
\beq
I\ge c|z|^2\int_0^1 |w(t)| \d t\ge \frac{c} 8 |z|^3. 
\eeq
Therefore, $F_{A,\alpha}$ is $3$-monotone. 
\end{proof}

More specific criteria are next, which are counterparts of Theorem \ref{Monothm4}.

\begin{theorem}\label{Monothm6}
Let $\alpha>0$. 
If
\beq\label{trB}
\frac{{\rm Tr}(A)}{|A|}\ge \sqrt{n-\frac{1}{(1+\alpha)^2}},
\eeq
then $F_{A,\alpha}$ is monotone.

If inequality \eqref{trB} is strict, then $F_{A,\alpha}$ is $(\alpha+2)$-monotone.
\end{theorem}
\begin{proof}
We solve for a solution $\lambda>0$ of the inequality
\beq
|A-\lambda I_n|^2\le \lambda^2/(1+\alpha)^2
\eeq
which is equivalent to 
\beq\label{Blam}
\left( n-\frac1{(1+\alpha)^2}\right) \lambda^2 -2{\rm Tr}(A)\lambda  +|A|^2\le 0.
\eeq

Under condition \eqref{trB}, the quadratic inequality \eqref{Blam} has a solution $\lambda>0$. Together with relation \eqref{nnorms}, we have \eqref{bbs}, and, thanks to Theorem \ref{Monothm5}(i), $F_{A,\alpha}$ is monotone.

When the inequality \eqref{trB} is strict, there is a solution $\lambda>0$ of the strict inequality in  \eqref{Blam}. Thanks to \eqref{nnorms} again, we have the strict inequality in \eqref{bbs}, and, by Theorem \ref{Monothm5}(ii), $F_{A,\alpha}$ is $(\alpha+2)$-monotone.
\end{proof}

\begin{lemma}\label{liplem}
Let $\alpha>0$ and $A$, $K$ be  $n\times n$ matrices.
If either $\alpha\ge 1$ or $K$ is invertible,
then there is a constant $C>0$ such that, for any $u,v\in \R^n$, 
\beq\label{Lip0}   \big ||Ku|^\alpha Au-|Kv|^\alpha Av\big |\le  C(|u|^\alpha+|v|^\alpha)|u-v|.
\eeq 
\end{lemma}
\begin{proof} First, we observe that  
\beq \label{KA1}
\begin{aligned}
        \big||Ku|^\alpha Au-|Kv|^\alpha Av\big |
        & \le |Ku|^\alpha |A(u-v)|+\big ||Ku|^\alpha -|Kv|^\alpha \big ||Av|\\
        &\le \|K\|_{\rm op}^\alpha  \|A\|_{\rm op}|u|^\alpha |u-v| 
        +\|A\|_{\rm op}\big ||Ku|^\alpha -|Kv|^\alpha \big ||v|.
    \end{aligned} 
\eeq  

\noindent\emph{Case $\alpha=1$.} Estimating $\big ||Ku| -|Kv|\big |\le |Ku-Kv|\le \|K\|_{\rm op}|u-v|$
in \eqref{KA1}, we obtain \eqref{Lip0} with $C=\|K\|_{\rm op}\|A\|_{\rm op}$.

\medskip
\noindent\emph{Case $\alpha> 1$.}
Using inequality \eqref{ele2}, we estimate$$\big ||Ku|^\alpha -|Kv|^\alpha \big |\le 2^{\alpha-1}( |Ku|^{\alpha-1}+|Kv|^{\alpha-1}) |Ku-Kv|
\le 2^{\alpha-1} \|K\|_{\rm op}^\alpha ( |u|^{\alpha-1}+|v|^{\alpha-1}) |u-v|.$$
Combining this with \eqref{KA1} yields 
\begin{align*}        
\big||Ku|^\alpha Au-|Kv|^\alpha Av\big |
&\le 2^{\alpha-1}\|K\|_{\rm op}^\alpha \|A\|_{\rm op} (|u|^\alpha +|u|^{\alpha-1}|v|+|v|^\alpha)|u-v|.
    \end{align*} 
By Young's inequality, we estimate $|u|^{\alpha-1}|v|\le |u|^\alpha+|v|^\alpha$, and obtain  \eqref{Lip0} with constant $C=2^\alpha \|K\|_{\rm op}^\alpha \|A\|_{\rm op}$.

\medskip
\noindent\emph{Case $\alpha\in(0,1)$ and $K$ is invertible.}    
Without  loss of generality, assume the line segment connecting $u$ and $v$ does not contain the origin.
Let $z=u-v$ and define, for $t\in [0,1]$, 
\beqs
w(t)=tu+(1-t)v,\quad h(t)=|Kw(t)|^\alpha Aw(t).
\eeqs
By The Fundamental Theorem of Calculus,
\beq
I\eqdef |Ku|^\alpha Au-|Kv|^\alpha Av=h(1)-h(0)=\int_0^1 h'(t)\d t.
\eeq
With  $w(t)\ne 0$ for all $t\in[0,1]$,  we have $Kw(t)\ne 0$, and 
\beqs 
    h'(t)
    =|Kw(t)|^\alpha Az +\alpha |Kw(t)|^{\alpha-2} (Kw(t)\cdot Kz) Aw(t).
\eeqs 
We estimate
\beq
    |h'(t)|
    \le  \|K\|_{\rm op}^\alpha \|A\|_{\rm op}  |w(t)|^\alpha |z|
    +\alpha \|K\|_{\rm op}\|A\|_{\rm op}  |Kw(t)|^{\alpha-1}|w(t)| |z|.
\eeq
With $\alpha-1<0$, we use inequality \eqref{op2} to have $|Kw(t)|^{\alpha-1}\le \|K^{-1}\|_{\rm op}^{1-\alpha}|w(t)|^{\alpha-1}$.
Thus, there is $C_0>0$ depending on $A,K,\alpha$ such that
$|h'(t)|\le C_0 |w(t)|^\alpha |z|.$
Therefore,
\begin{align*}
|I|&\le \int_0^1 |h'(t)|\d t\le C_0  |z|  \int_0^1 |w(t)|^\alpha\d t \\
&\le  C_0 |z|  \int_0^1  t^\alpha |u|^\alpha +(1-t)^\alpha |v|^\alpha \d t
=\frac{C_0}{\alpha+1}  (|u|^\alpha + |v|^\alpha) |z| ,
\end{align*}
which  proves \eqref{Lip0}.
\end{proof}

\begin{remark}\label{liprmk}
In the case $\alpha\in(0,1)$ and $K$ is not invertible, we only have 
\beq\label{Lip1}
   \big ||Ku|^\alpha Au-|Kv|^\alpha Av\big |\le 
   \|K\|_{\rm op}^\alpha \|A\|_{\rm op}(|u|^\alpha |u-v| +|v||u-v|^\alpha).
    \eeq    
This is weaker than \eqref{Lip0} when $|u-v|$ is small.    
Inequality \eqref{Lip1}, in fact, is a consequence of \eqref{KA1} and 
the following estimate, thanks to inequality \eqref{ele1},  
$$\big ||Ku|^\alpha -|Kv|^\alpha \big |\le |Ku-Kv|^\alpha \le \|K\|_{\rm op}^\alpha |u-v|^\alpha.$$
\end{remark}

\medskip
Consider an integer $N\ge 0$,  real numbers $0\le \alpha_0<\alpha_1<\alpha_2<\ldots<\alpha_N$,
and $n\times n$ matrices $\bar A_0,\bar A_1,\ldots,\bar A_N$.
For $0\le k\le N$, let $F_k=F_{\bar A_k,\alpha_k}$ or $F_k=\widetilde F_{\bar A_k,\alpha_k}$. 

Let  $B$ be a trilinear mapping  from $\R^n\times \R^n\times\R^n$ to $\R^n$ as in \eqref{Bnd}.

Define the mapping $F:\R^n\to\R^n$ by 
\begin{align}
\label{reF1}
F(u)&=G_B(u), \text{ or }\\
\label{reF2}
F(u)&=\sum_{k=0}^N F_k(u), \text{ or }\\
\label{reF3}
F(u)&=\sum_{k=0}^N F_k(u)+ G_B(u).
\end{align}

Note that if $\alpha_0=0$ then we have the (anisotropic) Darcy term $\bar A_0u$ in \eqref{reF2} and \eqref{reF3}.

\begin{proposition}\label{thmix}
Let $F$ be as in \eqref{reF1}--\eqref{reF3}. 
Assume the followings.
\begin{enumerate}[label=\rnum]    
    \item \label{acond} In the case \eqref{reF1}, the function $G_B$ is $3$-monotone. 

    \item\label{bcond} In the case \eqref{reF2}, each function $F_k$ is monotone for $0\le k\le N-1$, and $F_N$ is $(\alpha_N+2)$-monotone with $\alpha_N>0$.

    \item\label{ccond} In the case \eqref{reF3}, all functions $F_k$ and $G_B$ are monotone. In addition,

    {\rm (c.1)} if $\alpha_N<1$, then $G_B$ is  $3$-monotone,

    {\rm (c.2)} if $\alpha_N>1$, then $F_N$ is  $(\alpha_N+2)$-monotone,

    {\rm (c.3)} if $\alpha_N=1$, then $F_N$ or $G_B$ is  $3$-monotone.
\end{enumerate}

Let 
$\beta=1$ in the case \eqref{reF1},
$\beta=\alpha_N$ in the case \eqref{reF2}, and 
$\beta=\max\{\alpha_N,1\}$ in the case \eqref{reF3}.
Then there are positive constants $c_1$ and $c_2$ such that, for all $u,v\in \R^n$,
\begin{align}
\label{F3} |F(u)-F(v)|&\le c_1(1+|u|^\beta+|v|^\beta)|u-v|,\\
\label{F4}  (F(u)-F(v))\cdot (u-v)&\ge c_2|u-v|^{\beta+2}.
\end{align}
\end{proposition}
\begin{proof}
In the case \eqref{reF1}, inequalities \eqref{F3} and \eqref{F4} come from \eqref{co}  and assumption \ref{acond}.

Consider cases \eqref{reF2} and \eqref{reF3}. 
In estimates below, $C_1$ and $C_2$ are positive constants.
Let $u,v$ be any vectors in $\R^n$. By  Lemma \ref{liplem}  and Young's inequality, we have
\beq \label{Fp3}
\begin{aligned}
\left |\sum_{k=0}^N (F_k(u)-F_k(v))\right|
&\le C_1 \sum_{k=0}^N (|u|^{\alpha_k}+|v|^{\alpha_k})|u-v|\\
&\le C_1(2N+1)(1+|u|^{\alpha_N}+|v|^{\alpha_N})|u-v|.
\end{aligned}
\eeq 
If $F_N$ is $(\alpha_N+2)$-monotone, then 
\beq \label{Fp4}
\left (\sum_{k=0}^N F_k(u)-\sum_{k=0}^N F_k(v)\right)\cdot (u-v) 
\ge (F_N(u)-F_N(v))\cdot (u-v)
\ge C_2|u-v|^{\alpha_N+2}.
\eeq 

Therefore, \eqref{Fp3} and \eqref{Fp4} prove \eqref{F3} and \eqref{F4} in the case \eqref{reF2}.

In the case \eqref{reF3}, combining \eqref{Fp3} with \eqref{co} and the using H\"older's inequality, we obtain \eqref{F3}. Finally, \eqref{F4} follows the fact $G_B$ is $3$-monotone for (c.1), or the fact \eqref{Fp4} for (c.2), or  one of these facts (or both) for (c.3).
\end{proof}
  
\section {Steady state flows}\label{StatProb}

We study the steady states for system \eqref{pseu} with a nonhomogeneous Dirichlet boundary condition. Redenoting $\m$ by $\uu$, and $\widetilde p$ by $p$, we particularly consider 
\beq\label{stationaryProb}
\left\{ 
\begin{aligned}
F(\uu)&=-\nabla p &&\text{ in }\Omega,\\
 \diver \uu&=f &&\text{ in }\Omega,\\
p&= \psi  &&\text{ on }\partial \Omega,
\end{aligned}
\right.
\eeq
where $f:\Omega\to \R$ and $\psi:\partial \Omega\to \R$ are given functions. The unknowns are $\uu:\Omega\to \R^n$ and $p:\Omega\to \R$.

Although our original equation has $F$ as in \eqref{genF} and \eqref{BB},
 we will study Problem \eqref{stationaryProb} with a more general form of $F$.
 Motivated by Proposition \ref{thmix}, we impose the following conditions on $F$.

\begin{assumption}\label{assuF}
The function $F:\R^n\to\R^n$ satisfies
$F(0)=0$, 
and there exist constants $s>2$ and $c_1,c_2>0$ such that, for all $x,y\in\R^n$,
\begin{align}
\label{F3a} |F(x)-F(y)|&\le c_1(1+|x|^{s-2}+|y|^{s-2})|x-y|,\\
\label{F4a}  (F(x)-F(y))\cdot (x-y)&\ge c_2 |x-y|^s.
\end{align}
\end{assumption}

Taking $y=0$ in \eqref{F3a} and \eqref{F4a}, we obtain, for all $x\in\R^n$,
\begin{align}
\label{F1s} |F(x)|&\le c_1 (|x|+|x|^{s-1}),\\
\label{F2s} F(x)\cdot x &\ge c_2|x|^s.
\end{align}

Note from \eqref{F3a} that $F$ is locally Lipschitz continuous.

We find a variational formulation for Problem \eqref{stationaryProb}. By using testing functions $\vv :\Omega\to \R^n$ for the first equation, and $q:\Omega\to \R$ for the second equation, integrating over $\Omega$, using formal integration by parts for the right-hand side of the first equations and utilizing the boundary condition in the third equation,  we formally obtain
\beq\label{WeakStationaryProb}
\left\{
\begin{aligned}
\int_\Omega F(\uu) \cdot \vv\,  \d x -\int_\Omega  p \, \diver\vv \, \d x&=- \int_{\partial\Omega}  \psi \vv\cdot \vec{\nu}  \, \d\sigma && \text{ for all }\vv\in C^\infty(\bar \Omega,\R^n),\\
\int_\Omega q \, \diver\uu \, \d x&=\int_\Omega f q \, \d x && \text{for all } q\in C^\infty(\bar \Omega).
\end{aligned}
\right.
\eeq

We need to find suitable functional spaces for solutions $(\uu,p)$ and testing funstions $\vv$, $q$.
Throughout this section, $r$ is the H\"older conjugate of $s$. Hence,
\beq
\label{a-const }
 r=\frac{s}{s-1}\in (1,2).
 \eeq 

We  define the spaces $V=\mathbf W_s({\rm div},\Omega)$ as in \eqref{Wsdef} and  $Z=L^r(\Omega)$.
Recall that their norms are $ \norm{\cdot}_{V}=\norm{\cdot}_{\mathbf W_s({\rm div},\Omega)}$ as in \eqref{Wsnorm}
and $\norm{\cdot}_Z=\norm{\cdot}_{0,r}$.

For the first integral in \eqref{WeakStationaryProb}, we define $a: \mathbf L^s(\Omega)\times \mathbf L^s(\Omega)\to \R$ by 
\beqs
a(\uu,\vv)=\int_\Omega F(\uu(x))\cdot \vv(x)\d x\text { for } \uu,\vv\in \mathbf L^s(\Omega). 
\eeqs

For the second and fourth integrals in \eqref{WeakStationaryProb}, we define a bilinear form $b:V\times Z\to \R$ by 
\beqs
b(\vv,q)=  \int_\Omega (\diver \vv) q \, \d x \text{ for }  \vv\in V, q\in Z.
\eeqs

\subsection*{Forcing function.} For the fifth integral in \eqref{WeakStationaryProb}, we assume $f\in L^s(\Omega)$ and define $\Phi_f:Z\to \R$ by
\beqs
\Phi_f(q)=\int_\Omega fq\, \d x\text{ for } q\in Z.
\eeqs
Then $\Phi_f\in Z'$ and 
\beq\label{Phinorm}  \norm{\Phi_f}_{Z'}=\norm{f}_{0,s}.
\eeq 

\subsection*{Boundary data} For the third integral in \eqref{WeakStationaryProb}, we assume  $\psi\in X_r$ and define $\mathcal T_\psi:V\to \R$ by
\beqs 
\mathcal T_\psi(\vv)=\inprod{\gamma_{{\rm n},s}(\vv),\psi}_{X_r',X_r} \text{ for }\vv\in V.
\eeqs 

Thanks to the Green's formula \eqref{green}, $\mathcal T_\psi(\vv)$ is the rigorous formulation for the boundary integral in \eqref{WeakStationaryProb}.
By \eqref{gamnorm}, one has $\mathcal T_\psi\in V'$ and 
\beq\label{VXnorms}
\norm{\mathcal T_\psi}_{V'}\le \bar c_1\norm{\psi}_{X_r}.
\eeq

\begin{definition}\label{weaksoln}
Given  $f\in L^s(\Omega)$ and $\psi\in X_r$,
a weak solution of Problem \eqref{stationaryProb} is a pair $(\uu,p)\in V\times Z$ that satisfies
\beq\label{equivform}
\left\{
\begin{aligned}
a(\uu, \vv) - b(\vv,p) &=-\mathcal T_\psi(\vv) && \text{for all }\vv\in V,\\
b(\uu,q)&=\Phi_f(q)  && \text{for all } q\in Z.
\end{aligned}
\right.
\eeq     
\end{definition}

We obtain  the existence, uniqueness, estimates and continuous dependence for weak solutions of  Problem~\eqref{stationaryProb} in the next theorem.

\begin{theorem}\label{mainthm}
The following statements hold true.

\begin{enumerate}[label=\tnum]

\item\label{t1} For any $f\in L^s(\Omega)$ and $\psi\in X_r$, there exists a  unique weak solution $(\uu,p)\in V\times Z$  of Problem \eqref{stationaryProb}.  

\item\label{t2} There is $c_3>0$ such that if  $f,\psi,(\uu,p)$ are as in part \ref{t1} then
\beq\label{udivp0}
\norm{\uu}_{0,s}+\norm{\diver \uu}_{0,s}+\norm{p}_{0,r}\le c_3\big (\norm{f}_{0,s}^{r-1}+\norm{f}_{0,s}^{s-1}
+\norm{\psi}_{X_r} ^{r-1}+ \norm{\psi}_{X_r} ).
\eeq

\item\label{t3} For $j=1,2$, assume $f_j,\psi_j,(\uu_j,p_j)$ are as in part \ref{t1}.
Let 
$$M_0=1+ \sum_{j=1}^2 (\norm{f_j}_{0,s}^{s-2}+\norm{\psi_j}_{X_r}^{2-r}).$$
Then 
\beq  \label{contpf1}
\begin{aligned}
&\|\uu_1-\uu_2\|_{0,s}+\|\diver(\uu_1-\uu_2)\|_{0,s}+\|p_1-p_2\|_{0,r}\\
&\le c_4 \left(\|f_1-f_2\|_{0,s} + M_0^r \|f_1-f_2\|_{0,s}^{r-1}+\|\psi_1-\psi_2\|_{X_r}+M_0\|\psi_1-\psi_2\|_{X_r}^{r-1}\right).
\end{aligned}
\eeq
Consequently,
\beq \label{contpf2}
\begin{aligned}
&\|\uu_1-\uu_2\|_{0,s}+\|\diver(\uu_1-\uu_2)\|_{0,s}+\|p_1-p_2\|_{0,r}
\le c_5 \left (1+ \sum_{i=1}^2 (\|f_i\|_{0,s}+\|\psi_i\|_{X_r})\right )^{r(s-2)}  \\
&\quad \times \left(\|f_1-f_2\|_{0,s}^{r-1}+\|f_1-f_2\|_{0,s}+\|\psi_1-\psi_2\|_{X_r}^{r-1}+\|\psi_1-\psi_2\|_{X_r}\right). 
\end{aligned}
\eeq 
Above, $c_4$ and $c_5$ are positive constants independent of $f_j,\psi_j,\uu_j,p_j$ for $j=1,2$.
\end{enumerate}
\end{theorem}

The proof of Theorem \ref{mainthm} will be presented in subsection \ref{mainpf} below.

\begin{remark}
Part \ref{t3} of Theorem \ref{mainthm} shows that the function that maps $(f,\psi)\in L^s(\Omega)\times X_r$ to the unique weak solution $(\uu,p)\in \mathbf W_s({\rm div},\Omega) \times L^r(\Omega)$ of Problem \eqref{stationaryProb} is locally H\"older continuous of order $(r-1)$, which belongs to $(0,1)$ thanks to property \eqref{a-const }.
\end{remark}

\subsection{Variational formulation on $V\times Z$}\label{subVZ}
Note that system  \eqref{equivform} has variational formulation on $V$ in its first equation and on $Z$ in its second equation. We will convert \eqref{equivform} to a variational formulation on $V\times Z$.

Suppose $(\uu,p)\in V\times Z$ is a solution of \eqref{equivform}. Then adding two equations in \eqref{equivform} gives
\beq \label{ahatint}
\widehat a( (\uu,p), (\vv,q)  ) =-\mathcal T_\psi(\vv)+\Phi_f(q)  \text{ for all }(\vv,q)\in V\times Z.
\eeq 
Here, $\widehat a$ is defined on $(V\times Z)^2$ by
\beqs
\widehat a( (\uu, p), (\vv,q)  )= a(\uu, \vv) - b(\vv, p) + b(\uu, q)\text{ for }(\uu,p),(\vv,q)\in V\times Z.
\eeqs

\begin{definition}
Denote ${\mathcal W}=V\times Z$ with the norm $\|(\uu,p)\|_{\mathcal W}=\|\uu\|_V+\|p\|_Z$ 
for $\uu\in V$ and $p\in Z$.     
\end{definition}

Because both $V$ and $Z$ are reflexive Banach spaces, so is $\mathcal W$.

\medskip

For any $\uu,\vv \in \mathbf L^s(\Omega)$, one has from property \eqref{F1s}  that
\beqs 
|a(\uu,\vv)|\le  \int_\Omega |F(\uu(x))\cdot \vv(x)|\d x
\le c_1\int_\Omega (|\uu(x)|+|\uu(x)|^{s-1}) |\vv(x)| \d x.
\eeqs 
Applying H\"older's inequality yields
\beq\label{a1}
\begin{aligned}
|a(\uu,\vv)|
&\le c_1 \left (\|1\|_{0,s/(s-2)}\|\uu\|_{0,s}+\|\uu\|_{0,s}^{s-1}\right ) \|\vv\|_{0,s}\\
&\le c_1(1+|\Omega|^\frac{s-2}{s}) \left (\|\uu\|_{0,s}+\|\uu\|_{0,s}^{s-1}\right ) \|\vv\|_{0,s}.    
\end{aligned}
\eeq 

For any $\vv\in V$ and $q\in Z$,  applying H\"older's inequality gives
\beq\label{b1}
|b(\vv,q)|\le \norm{\diver \vv}_{0,s}\norm{p}_{0,r} .
\eeq 

For each $(\uu,p)\in \mathcal W$, the mapping $(\vv,q)\in\mathcal W\mapsto \widehat a( (\uu,p), (\vv,q)  ) $ is linear, and one has, following \eqref{a1}, \eqref{b1} and \eqref{Wadnorm}, 
\begin{align*}
|\widehat a( (\uu,p), (\vv,q)  ) |
&\le c_1(1+|\Omega|^\frac{s-2}{s}) \left ( \|\uu\|_{0,s}+\|\uu\|_{0,s}^{s-1}\right ) \|\vv\|_{0,s} 
+ \norm{\diver \vv}_{0,s}\norm{p}_{0,r} 
+\norm{\diver \uu}_{0,s}\norm{q}_{0,r}\\
&\le \max\left\{1,c_1(1+|\Omega|^\frac{s-2}{s})\right\} ( \|\uu\|_{0,s}+\|\uu\|_{0,s}^{s-1} + \norm{p}_{0,r} 
+\norm{\diver \uu}_{0,s})\norm{(\vv,q)}_{\mathcal W}.
\end{align*}
Therefore, one can define a mapping $\mathcal A: \mathcal W \to \mathcal W'$  by 
\beqs
\intb{\mathcal A(\uu,p), (\vv,q)}_{\mathcal W',\mathcal W}=\widehat a( (\uu, p), (\vv,q)  )
\text{ for }(\uu,p),(\vv,q)\in\mathcal W.
\eeqs

\medskip
For the right-hand side of \eqref{ahatint}, we define a linear functional $\mathcal F_{\psi,f}$ on $\mathcal W$  by 
\beq \label{Fdual}
\mathcal F_{\psi,f} (\vv, q) = -\mathcal T_\psi(\vv)+\Phi_f(q), 
\text{ for }(\vv,q)\in \mathcal W.
\eeq 
For any $(\vv,q)\in \mathcal W$, one has  from \eqref{Fdual} that
$$|\mathcal F_{\psi,f} (\vv, q)| \le \norm{\mathcal T_\psi}_{V'}\norm{\vv}_V + \norm{\Phi_f}_{Z'}\norm{q}_{Z}
\le (\norm{\mathcal T_\psi}_{V'} + \norm{\Phi_f}_{Z'}) \norm{(\vv,q)}_{\mathcal W}. 
$$
Thus, $\mathcal F_{\psi,f}\in \mathcal W'$ and 
we can rewrite \eqref{ahatint} as\beq\label{VQform}
\mathcal A(\uu,p)=\mathcal F_{\psi,f}.
\eeq

Now, suppose $(\uu,p)\in \mathcal W$ is a solution of \eqref{VQform}.
By taking $q=0$ and $\vv\in V$ arbitrary in \eqref{ahatint}, we obtain the first equation of \eqref{equivform}.
By taking $\vv=0$ and $q\in Z$ arbitrary in \eqref{ahatint}, we obtain the second equation of \eqref{equivform}.
Therefore, $(\uu,p)$ is a solution of \eqref{equivform}.
We have proved that
\beq\label{equeq}
   \emph{ problems \eqref{equivform} and \eqref{VQform} are equivalent.}
\eeq 

Equation \eqref{VQform}, in turn, is a special case of a more general class of equations
\beq\label{VQgen}
\mathcal A(\uu,p)=\mathcal F,\quad \mathcal F\in\mathcal W'.
\eeq
We will study \eqref{VQgen} and apply the findings to \eqref{VQform}.

\begin{definition}
    For $\mathcal F\in \mathcal W'$, define $\pi_1\mathcal F:V\to \R$ and $\pi_2\mathcal F:Z\to\R$ by
    $$(\pi_1\mathcal F)(\vv)=\mathcal F(\vv,0)\text{ for $\vv\in V$ and } 
    (\pi_2\mathcal F)(q)=\mathcal F(0,q) \text{ for $q\in Z$.}$$
\end{definition}

Then  $\pi_1:\mathcal W'\to  V'$ and $\pi_2:\mathcal W' \to Z'$ are linear,  and
\beq \label{piVZ}
\|\pi_1\mathcal F\|_{V'},\|\pi_2\mathcal F\|_{Z'}\le \|\mathcal F\|_{\mathcal W'}.
\eeq 

One has, for all $(\vv,q)\in\mathcal W$, 
$\mathcal F(\vv,q)=(\pi_1\mathcal F)(\vv)+(\pi_2\mathcal F)(q),$ which implies
\beq\label{mFuv} |\mathcal F(\vv,q)|\le \|\pi_1\mathcal F\|_{V'}\|\vv\|_V+\|\pi_2\mathcal F\|_{Z'}\|q\|_Z,
\eeq
and, consequently,
$  \|\mathcal F\|_{\mathcal W'}\le \|\pi_1\mathcal F\|_{V'}+\|\pi_2\mathcal F\|_{Z'}$. 

Similar to \eqref{equeq}, we have the following equivalence.

\begin{lemma}\label{genequiv}
Equation \eqref{VQgen} is equivalent to
\beq\label{genproj}
\left\{
\begin{aligned}
a(\uu, \vv) - b(\vv,p) 
&=(\pi_1\mathcal F)(\vv)  &&  \text{ for all } \vv\in V,\\
b(\uu,q)
&=(\pi_2 \mathcal F)(q) &&  \text{ for all } q\in Z.
\end{aligned}
\right.
\eeq
\end{lemma}

When $\mathcal F=\mathcal F_{\psi,f}$, we have
\beq\label{ppsf}
\pi_1(\mathcal F_{\psi,f})=-\mathcal T_\psi,\quad \pi_2(\mathcal F_{\psi,f})=\Phi_f,
\eeq
and system \eqref{genproj} becomes \eqref{equivform}.

\medskip
In studying \eqref{VQgen} and \eqref{genproj}, the coercivity of $a(\cdot,\cdot)$ and $\mathcal A$ plays an important role.
For any $\uu \in \mathbf L^s(\Omega)$, using property \eqref{F2s}, we have
\beq\label{a2}
a(\uu,\uu)=\int_\Omega  F(\uu(x))\cdot \uu(x)\d x 
\ge c_2\int_\Omega  |\uu(x)|^s\d x=c_2 \|\uu\|_{0,s}^s.
\eeq 
Thus, $a(\cdot,\cdot)$ is coercive on $\mathbf L^s(\Omega)$. 
For $\mathcal A$ to be coercive, it must satisfy
\beq\label{coerdef}
\lim_{\norm{(\uu,p)}_{\mathcal W}\to \infty} \frac{\intb{\mathcal A(\uu,p), (\uu,p) }_{\mathcal W',\mathcal W}}{\norm{(\uu,p)}_{\mathcal W}}=\infty.
\eeq
However, for any $(\uu,p)\in \mathcal W$, 
\beq\label{Aup}
\intb{\mathcal A(\uu,p), (\uu,p) }_{\mathcal W',\mathcal W}
=\widehat a((\uu,p),(\uu,p))=a(\uu, \uu).
\eeq
Hence, \eqref{coerdef} fails and $\mathcal A$ is not coercive on $\mathcal W$.
To overcome this issue, we need some regularized problem.

\subsection{The regularized problem} \label{subreg}
For $\uu,\vv\in V$ and $p,q\in Z$, define
\beq\label{RRdef}
R_1(\uu,\vv) =   \int_\Omega|\diver \uu|^{s-2} \diver \uu\cdot  \diver \vv\, \d x 
\text{ and }
R_2(p,q)=  \int_\Omega |p|^{r-2}p\cdot  q \, \d x.
\eeq

For $\varep>0$, we consider the following regularized problem of \eqref{VQgen} and \eqref{genproj}: Find $(\uu,p)\in V\times Z$ such that
\beq\label{reg-prob}
\left\{
\begin{aligned}
a(\uu, \vv)+ \varep R_1(\uu,\vv)  - b(\vv,p) 
&=(\pi_1 \mathcal F)(\vv)  &&  \text{ for all } \vv\in V,\\
\varep R_2(p,q) + b(\uu,q)
&=(\pi_2\mathcal F)(q) &&  \text{ for all } q\in Z.
\end{aligned}
\right.
\eeq
Adding the equations in \eqref{reg-prob} yields 
\beq\label{VQreg}
\widehat a_\varep((\uu,p), (\vv,q) )=\mathcal F(\vv,q) \text{ for all }(\vv,q)\in V\times Z.
\eeq 
Here, $\widehat a_\varep$ is defined on $(V\times Z)^2$ by
\beq\label{a-eps}
\begin{split}
\widehat a_\varep((\uu,p), (\vv,q) ) = \widehat a( (\uu, \vv), (\vv,q)  )  + \varep R_1(\uu,\vv) 
+\varep R_2(p,q) \text{ for }(\uu,p),(\vv, q)\in V\times Z.
\end{split}
\eeq

Note, for $\uu,\vv\in V$ and $p,q\in Z$, that
\beq\label{RRineq}
|R_1(\uu,\vv)|\le \norm{\diver \uu}_{0,s}^{s-1}\norm{\diver \vv}_{0,s},\quad 
|R_2(p,q)|\le \norm{p}_{0,r}^{r-1}\norm{q}_{0,r},
\eeq
and
\begin{align*}
&|\widehat a_\varep((\uu,p), (\vv,q) )|
\le \norm{\mathcal A(\uu,p)}_{\mathcal W'}\norm{(\vv,q)}_{\mathcal W}
+ \varep \norm{\diver \uu}_{0,s}^{s-1} \norm{\diver \vv}_{0,s}
+ \varep \norm{p}_{0,r}^{r-1} \norm{q}_{0,r}\\
&\le \left (\norm{\mathcal A(\uu,p)}_{\mathcal W'}+\varep\norm{(\uu,p)}_{\mathcal W}^{s-1}+\varep\norm{(\uu,p)}_{\mathcal W}^{r-1}\right ) \cdot \norm{(\vv,q)}_{\mathcal W}.
\end{align*}
Then, similar to the definition of $\mathcal A$, there is a unique mapping $\mathcal A_\varep: \mathcal W\to \mathcal W'$ so that 
\beqs
\intb{\mathcal A_\varep(\uu,p), (\vv,q) }_{\mathcal W',\mathcal W} =\widehat  a_\varep((\uu,p), (\vv,q)) \text{ for all } (\uu,p), (\vv,q)\in \mathcal W. 
\eeqs

Similar to Lemma \ref{genequiv}, we have the following.

\begin{lemma}\label{epseqv} Problem \eqref{reg-prob} is equivalent to 
\beq \label{dualeq}
\mathcal A_\varep (\uu, p) = \mathcal F.
\eeq
\end{lemma}

Together with \eqref{a1} and \eqref{a2}, the following are the basic properties of $a(\cdot,\cdot)$.
\begin{lemma}\label{alem}
For any $\uu,\vv,\ww \in \mathbf L^s(\Omega)$, one has 
\begin{align} 
|a(\uu,\ww)-a(\vv,\ww) |
&\le c_6 \left(1+\norm{\uu}_{0,s}^{s-2}+ \norm{\vv}_{0,s}^{s-2}\right)\norm{\uu-\vv}_{0,s}\norm{\ww}_{0,s}, \label{a3}
\\
\label{a4} 
a(\uu,\uu-\vv)-a(\vv,\uu-\vv) 
&\ge c_2 \norm{\uu-\vv}_{0,s}^s,
\end{align}
where $c_6$ is a positive constant, and $c_2$ is from Assumption \ref{assuF}. 
\end{lemma}
\begin{proof}
Using property  \eqref{F3a}, we have 
\begin{align*}
|a(\uu,\ww)-a(\vv,\ww) |
&\le \int_\Omega |F(\uu(x))-F(\vv(x))|\, |\ww(x)|\d x
\le c_1\int_\Omega   (1+|\uu|^{s-2}+|\vv|^{s-2})|\uu-\vv||\ww|\d x .
\end{align*}
Applying H\"older's inequality to three functions and powers $s/(s-2),s,s$ gives
\begin{align*}
|a(\uu,\ww)-a(\vv,\ww) |
\le c_1 \left(\norm{1}_{0,s}^{s-2}+\norm{\uu}_{0,s}^{s-2}+ \norm{\vv}_{0,s}^{s-2}\right)\norm{\uu-\vv}_{0,s}\norm{\ww}_{0,s}.
\end{align*}
Thus, we obtain \eqref{a3} with $c_6=c_1\max\{1,|\Omega|^{\frac{s-2}s}\}$.
By \eqref{F4a}, we have
\begin{align*}
    a(\uu,\uu-\vv)-a(\vv,\uu-\vv
    )=  \int_\Omega (F(\uu)-F(\vv))\cdot (\uu-\vv)\d x
\ge c_2\int_\Omega |\uu-\vv|^s\d x,
\end{align*}
which proves \eqref{a4}.
\end{proof}

\begin{proposition}\label{StatSol}
For every $\varep>0$ and $\mathcal F\in\mathcal W'$, there exists a unique solution $(\m_\varep, p_\varep)\in V\times Z$ of the regularized  problem \eqref{dualeq}.  
\end{proposition}
\begin{proof}
Below, we establish that  
$\mathcal A_\varep$ is continuous, coercive and strictly monotone.

\medskip\noindent
\emph{Proof of the fact $\mathcal A_\varep$ is continuous.} For any $(\uu_1,p_1),(\uu_2,p_2),(\vv,q)\in V\times Z$, we have
\beq\label{est0}
\begin{aligned}
J&\eqdef \intb{\mathcal A_\varep(\uu_1,p_1)-\mathcal A_\varep(\uu_2,p_2), (\vv,q) }_{\mathcal W',\mathcal W}\\
&=J_1+\varep J_2+\varep J_3- b(\vv,p_1-p_2) +b(\uu_1-\uu_2,q),
\end{aligned}
\eeq
where
\begin{align*}
    J_1&=a(\uu_1,\vv)-a(\uu_2,\vv) ,\\
J_2&=R_1(\uu_1,\vv)-R_1(\uu_2,\vv)=\int_\Omega (|\diver \uu_1|^{ s-2} \diver \uu_1-|\diver \uu_2|^{ s-2} \diver \uu_2)\diver\vv \, \d x,\\
J_3&=R_2(p_1,q)-R_2(p_2,q)=\int_\Omega (|p_1|^{r-2}p_1-|p_2|^{r-2}p_2)q\, \d x.
\end{align*}

By  \eqref{a3}, 
\begin{align*}
|J_1|
&\le c_6 \left(1+\norm{\uu_1}_{0,s}^{s-2}+ \norm{\uu_2}_{0,s}^{s-2}\right)\norm{\uu_1-\uu_2}_{0,s}\norm{\vv}_{0,s}\\
&\le c_6 \left(1+\norm{\uu_1}_V^{s-2}+ \norm{\uu_2}_V^{s-2}\right)\norm{\uu_1-\uu_2}_V\norm{\vv}_V.
\end{align*}

By \eqref{inq1} and H\"older's inequality we find that   
\beqs
\begin{split}
|J_2|
&\le 2^{s-2}\int_\Omega (|\diver \uu_1|^{ s-2}+|\diver \uu_2|^{ s-2})\cdot |\diver (\uu_1- \uu_2)|\cdot |\diver\vv| \d x\\
&\le 2^{s-2}\big(\norm{\diver \uu_1}_{0,s}^{s-2}+\norm{\diver \uu_2}_{0,s}^{s-2}\big) \norm{\diver (\uu_1- \uu_2)}_{0,s}\norm{\diver\vv}_{0,s}\\
&\le 2^{s-2}\big(\norm{\uu_1}_V^{s-2}+\norm{\uu_2}_V^{s-2}\big) \norm{\uu_1- \uu_2}_V\norm{\vv}_V.
\end{split}
\eeqs

By \eqref{inq2}, using H\"older's inequality, and noting that $r-2\in(-1,0)$, we have
\begin{align*}
|J_3|&\le \int_\Omega \left | |p_1|^{r-2}p_1-|p_2|^{r-2}p_2\right | \cdot|q| \d x
\le 2^{2-r}\int_\Omega |p_1-p_2|^{r-1} |q| \d x \le C\norm{p_1-p_2}_{0,r}^{r-1}\norm{q}_{0,r}\\
&\le 2^{2-r}\norm{p_1-p_2}_Z^{r-1}\norm{q}_Z.
\end{align*}

By  \eqref{b1}, 
\begin{align*}
|b(\vv,p_1-p_2)| +|b(\uu_1-\uu_2,q)| 
&\le \norm{\diver\vv}_{0,s}\norm{p_1-p_2}_{0,r}+ \norm{\diver(\uu_1-\uu_2)}_{0,s}\norm{q}_{0,r}\\
&\le \norm{p_1-p_2}_Z\norm{\vv}_V+ \norm{\uu_1-\uu_2}_V\norm{q}_Z.
\end{align*}

From \eqref{est0} and the above estimates, it follows that
\begin{align*}
|J|
&\le  C_1(1+\varep)\big( 1+\norm{\uu_1}_V^{s-2}+\norm{\uu_2}_V^{s-2}\big)\\
&\quad  \times \big(\norm{\uu_1-\uu_2}_{V}+\norm{p_1-p_2}_Z +\norm{p_1-p_2}_Z^{r-1} \big)\big( \norm{\vv}_{V}+ \norm{q}_Z\big),
 \end{align*}
where $C_1=c_6+2^{s-2}+2^{2-r}+1$.
This yields 
\begin{align*}
\norm{\mathcal A_\varep(\uu_1,p_1)-\mathcal A_\varep(\uu_2,p_2)}_{\mathcal W'}
&\le C_1(1+\varep) \big( 1+\norm{\uu_1}_V^{s-2}+\norm{\uu_2}_V^{s-2}\big)\\
&\quad \times \big(\norm{\uu_1-\uu_2}_{V}+\norm{p_1-p_2}_Z +\norm{p_1-p_2}_Z^{r-1} \big).
\end{align*}
Thus, $\mathcal A_\varep$ is continuous.

\medskip\noindent
\emph{Proof of the fact $\mathcal A_\varep$ is coercive.} 
We need to prove \eqref{coerdef} for $\mathcal A_\varep$ replacing $\mathcal A$.
For any $(\uu,p)\in \mathcal W$, we have from \eqref{a2}, \eqref{Aup}, \eqref{RRdef} and \eqref{a-eps},  that
\beqs
\begin{split}
\intb{\mathcal A_\varep(\uu,p), (\uu,p) }_{\mathcal W',\mathcal W}
&=\widehat a((\uu,p),(\uu,p))+ \varep R_1(\uu,\uu) + \varep R_2(p,p)
\ge c_2 \norm{\uu}_{0,s}^{s}+\varep  \norm{\diver\uu}_{0,s}^{ s}  +\varep  \norm{p}_{0,r}^{r}\\
&\ge\min\{c_2,\varep\}\big(  \norm{\uu}_V^{ s}  +\norm{p}_Z^{ r}\big).
\end{split}
\eeqs  

Note that $s> 2 >r>1$. We consider $\norm{\uu}_V+\norm{p}_Z\ge 2$. 
If $\norm{\uu}_V \ge 1$ then
\beqs
\norm{\uu}_V^{ s}  +\norm{p}_Z^{r}
\ge \norm{\uu}_V^r  +\norm{p}_Z^r\ge  2^{1-r} (\norm{\uu}_V+ \norm{p}_Z)^r.
\eeqs
If $\norm{\uu}_V < 1$, then $\norm{p}_Z>1>\norm{\uu}_V$, and 
\beqs
\norm{\uu}_V^{ s}  +\norm{p}_Z^{r}
\ge \norm{p}_Z^r
\ge  \left ( \frac12 \norm{p}_{Z} + \frac12\norm{\uu}_V\right )^r=2^{-r} (\norm{\uu}_V+ \norm{p}_Z)^r.
\eeqs
In both cases, we find that 
\beqs
\frac{\intb{\mathcal A_\varep(\uu,p), (\uu,p) }_{\mathcal W',\mathcal W}}{\norm{(\uu,p)}_{\mathcal W} }
\ge 2^{-r}\min\{c_2,\varep\} \norm{(\uu,p)}_{\mathcal W}^{r-1}\to \infty
\text{ as $\norm{(\uu,p)}_{\mathcal W} \to  \infty$.}
\eeqs
Therefore,  $\mathcal A_\varep$ is coercive.

\medskip\noindent
\emph{Proof of the fact $\mathcal A_\varep$ is strictly monotone.} 
By ``strictly monotonotone", we mean that
$$\intb{\mathcal A_\varep (\uu,p)-\mathcal A_\varep(\vv,q), (\uu,p)-(\vv,q) }_{\mathcal W',\mathcal W}>0$$
for all $(\uu,p),(\vv, q)\in \mathcal W$ with $(\uu,p)\neq(\vv, q)$. 

Let $(\uu,p),(\vv,q)\in \mathcal W$. We have 
\beqs
\begin{split}
I&\eqdef \intb{\mathcal A_\varep (\uu,p)-\mathcal A_\varep(\vv,q), (\uu-\vv,p-q) }_{\mathcal W',\mathcal W}\\
&=a(\uu,\uu-\vv)-a(\vv,\uu-\vv)
+ \varep(R_1(\uu,\uu-\vv)-R_1(\vv,\uu- \vv))+ \varep (R_2(p,p-q)-R_2(q,p-q)). 
\end{split}
\eeqs

 By \eqref{a4},  we have 
 \beqs
 a(\uu,\uu-\vv)-a(\uu,\uu-\vv)
\ge c_2\norm{\uu-\vv}_{0,s}^s. 
 \eeqs   

  By \eqref{inq3}, we have 
 \begin{align*} 
R_1(\uu,\uu-\vv)-R_1(\vv,\uu- \vv)
&= \int_\Omega(|\diver \uu|^{s-2} \diver \uu-|\diver \vv|^{s-2} \diver \vv)\cdot(\diver \uu-  \diver \vv) \, \d x \\
&\ge  2^{1-s}\norm{\diver (\uu-\vv)}_{0,s}^{s}. 
 \end{align*}  
 
Utilizing inequality \eqref{inq4}, we have
\begin{align*} 
R_2(p,p-q)-R_2(q,p-q)
=\int_\Omega (|p|^{r-2}p-|q|^{r-2}q)\cdot(p-q) \, \d x
\ge (r-1)\int_\Omega  (|p|+|q|)^{r-2} |p-q|^2 \, \d x. 
\end{align*}

Thus, combining these estimates yields
\beqs 
I
\ge C_2\left( \norm{\uu-\vv}_{0,s}^s + \varep\norm{\diver (\uu-\vv)}_{0,s}^{ s}+\varep\int_\Omega  (|p|+|q|)^{r-2} |p-q|^2 \, \d x\right),
\eeqs 
where $C_2=\min\{c_2,2^{1-s},r-1\}$.
This implies  that $I$ is positive  whenever $(\uu,p)\neq(\vv, q)$.   
Therefore, $\mathcal A_\varep$ is strictly monotone.

\medskip
By the Browder--Minty Theorem \cite[Theorem 26.A]{ZeidlerIIB}, there exists a unique solution $(\uu_\varep,p_\varep)\in \mathcal W$ of equation \eqref{dualeq}. 
\end{proof}

Next, we obtain a uniform upper bound for solutions $(\uu_\varep, p_\varep)$ of \eqref{dualeq} when $\varep$ is  sufficiently small.

\begin{lemma}\label{stationary-sol-indep-eps}
There exists $\varep_0>0$  such that if $\varep\in(0,\varep_0)$, $\mathcal F\in\mathcal W'$,  and  $(\uu_\varep, p_\varep)$ is the unique solution of \eqref{dualeq}, then 
\beq\label{uQmV}
\norm{(\uu_\varep,p_\varep)}_{\mathcal W}\le \mathcal C,
\eeq
where number $\mathcal C\ge 0$ is independent of $\varep$ and depends increasingly on $\norm{\mathcal F}_{\mathcal W'}$.
\end{lemma}
\begin{proof}
Thanks to Lemma \ref{epseqv}, $(\uu_\varep, p_\varep)$ is the unique solution of the equivalent system \eqref{reg-prob}. 
Denote $\beta=\norm{\pi_1 \mathcal F}_{V'}$ and $\gamma=\norm{\pi_2 \mathcal F}_{Z'}$.
Let $\varep>0$.
In calculations below, $C$ denotes a positive constant independent of $\varep$, $\beta$, $\gamma$ that may have varied values from one line to another, while $C_i$, for $i=0,1,2,3$, denotes a positive constant independent of $\varep$, $\beta$, $\gamma$ with a fixed value.

Choosing  $q= |\diver \uu_\varep|^{s-2}\diver \uu_\varep\in Z$ in the second equation of \eqref{reg-prob}, noticing that $\norm{q}_Z=\norm{\diver \uu_\varep}_{0,s}^{s- 1}$, and using H\"older's inequality, we find that
\beqs
\norm{\diver \uu_\varep}_{0,s}^{s} 
\le  \gamma \norm{\diver \uu_\varep}_{0,s}^{s- 1} +\varep \norm{ p_\varep}_{0,r}^{r-1}\norm{\diver \uu_\varep}_{0,s}^{s- 1}.
\eeqs
Consequently, 
\beq\label{bound-divm}
\norm{\diver \uu_\varep}_{0,s} \le \gamma  +\varep \norm{ p_\varep}_{0,r}^{r-1}.
\eeq

Taking $(\uu, p)=(\vv, q)=(\uu_\varep,  p_\varep)$ in \eqref{VQreg} and using \eqref{mFuv}, \eqref{Wadnorm}  give
\beq\label{estRHS1}
\begin{aligned}
a(\uu_\varep, \uu_\varep)+\varep \|\diver \uu_\varep\|_{0,s}^s + \varep\|p_\varep\|_{0,r}^r  
=\mathcal F(\uu_\varep,p_\varep)  
\le \beta \big(\norm{\uu_\varep}_{0,s} + \norm{\diver \uu_\varep}_{0,s} \big)+\gamma \norm{p_\varep}_{0,r} .
\end{aligned}
\eeq

Applying inequality \eqref{a1} to the first term of \eqref{estRHS1}, neglecting the next two terms, and utilizing the estimate \eqref{bound-divm} for the last divergence term,  we obtain  
\beqs
c_2\norm{\uu_\varep}_{0,s}^{s}
\le  \beta \big(\norm{\uu_\varep}_{0,s} +\gamma  + \varep \norm{ p_\varep}_{0,r}^{r-1}\big)+\gamma \norm{p_\varep}_{0,r}. 
\eeqs
By Young's inequality, specifically, the last one in \eqref{Yineq}, one has $\beta \norm{\uu_\varep}_{0,s}\le  (c_2/2)\norm{\uu_\varep}_{0,s}^s+C\beta^r$. 
It follows that
\beq\label{bound-mvarep}
\norm{\uu_\varep}_{0,s}^{s}
\le  C\left (\beta^r + \beta \gamma  + \varep \beta \norm{ p_\varep}_{0,r}^{r-1}+\gamma \norm{p_\varep}_{0,r}\right ). 
\eeq

Next, we use \eqref{supinfcdn} to bound $ \norm{p_\varep}_{0,r}$ and close the estimates \eqref{bound-divm} and \eqref{bound-mvarep}.
Using $q=p_\varep$ in the first equation of \eqref{reg-prob} and the estimate of $\norm{ \diver\uu_\varep}_{0,s}$ in \eqref{bound-divm} above, we find that 
\beqs
\begin{split}
b(\vv, p_\varep)
&=a(\uu_\varep,\vv) +\varep R_1( \uu_\varep,\vv)- (\pi_1 \mathcal F)\vv\\
&\le C(\norm{\uu_\varep}_{0,s}+\norm{\uu_\varep}_{0,s}^{s-1} ) \norm{ \vv}_{0,s}+\varep\norm{\diver \uu_\varep}_{0,s}^{s-1}\norm{\diver \vv}_{0,s} +\beta \norm{\vv}_V\\
&\le \left[C\big(\norm{\uu_\varep}_{0,s}+\norm{\uu_\varep}_{0,s}^{s-1} \big)+\varep\big( \gamma +\varep \norm{ p_\varep}_Z^{r-1}\big)^{s - 1} +\beta\right] \norm{\vv}_V.
\end{split}
\eeqs
Applying inequality \eqref{supinfcdn} yields
\beq\label{ppZ}
\norm{ p_\varep}_Z
\le  C_*\sup_{\vv\in V\backslash\{0\}} \frac{b(\vv, p_\varep)}{\norm{\vv}_{V}}
\le C_0\left(\norm{\uu_\varep}_{0,s}+\norm{\uu_\varep}_{0,s}^{s-1}+\varep\gamma ^{s-1} +\varep^s  \norm{p_\varep}_Z+  \beta\right).
\eeq 

Set $\varep_0=\min\{1,(2C_0)^{-1/s}\}$ and consider $\varep\in(0,\varep_0)$.
We obtain  from \eqref{ppZ} that
\beq\label{bound-rhovarep}
\norm{ p_\varep}_Z\le 2C_0\big(\norm{\uu_\varep}_{0,s}+\norm{\uu_\varep}_{0,s}^{s-1} +\varep \gamma ^{s-1}
+  \beta\big). 
\eeq  
Consequently, noting that $(r-1)(s-1)=1$, one has
\beq\label{boundrr}
\norm{ p_\varep}_Z^{r-1}\le C\big(\norm{\uu_\varep}_{0,s}^{r-1}+\norm{\uu_\varep}_{0,s}+\varep ^{r-1}\gamma +  \beta^{r-1}\big). 
\eeq  

Combining \eqref{bound-rhovarep} and \eqref{boundrr} with \eqref{bound-mvarep} leads to 
\begin{align*}
 \norm{\uu_\varep}_{0,s}^{s} 
&\le C\beta^r +C \beta \gamma 
 + C\varep \beta\big(\norm{\uu_\varep}_{0,s}^{r-1}+\norm{\uu_\varep}_{0,s}+\varep ^{r-1}\gamma +  \beta^{r-1}\big)\\
&\quad +C\gamma  \big(\norm{\uu_\varep}_{0,s}+\norm{\uu_\varep}_{0,s}^{s-1} +\varep \gamma ^{s-1}+  \beta\big)\\ 
&\le CJ + C\varep \beta\big(\norm{\uu_\varep}_{0,s}^{r-1}+\norm{\uu_\varep}_{0,s}\big)
+C\gamma  \big(\norm{\uu_\varep}_{0,s}+\norm{\uu_\varep}_{0,s}^{s-1}\big),
\end{align*}
where $J=(1+\varep)\beta^r +(1+\varep^r) \beta \gamma 
+\varep \gamma ^{s}$.

Let $s_*=s/(s-r+1)=s(s-1)/(s(s-1)-1)\in(1,r)$. 
Then using Young's inequality we obtain
\begin{align*}
 \norm{\uu_\varep}_{0,s}^{s} 
&\le CJ+\frac12 \norm{\uu_\varep}_{0,s}^{s}
+ C \big(\varep^{s_*} \beta^{s_*}+\varep^r \beta^r\big)
+C(\gamma^r+ \gamma^{s}\big).
\end{align*}
This implies
\beq \label{mLsbound}
 \norm{\uu_\varep}_{0,s}\le C_1 d_1(\varep,\beta,\gamma),
\eeq 
where 
\beq 
d_1(\varep,\beta,\gamma)
=\Big[ (1+\varep^r)\beta^r+ \varep^{s_*} \beta^{s_*} +(1+\varep^r) \beta \gamma+\gamma^r +(1+\varep) \gamma^{s} \Big ]^{1/s}.
\eeq 

Utilizing estimate \eqref{mLsbound} in \eqref{bound-rhovarep} yields 
\begin{align}\label{ub1}
\norm{ p_\varep}_Z
\le C_2 d_2(\varep,\beta,\gamma),
\text{ where } 
d_2(\varep,\beta,\gamma)=d_1(\varep,\beta,\gamma)+d_1(\varep,\beta,\gamma)^{s-1} +\varep \gamma^{s-1} +  \beta.
\end{align}
Combining \eqref{ub1} and  \eqref{bound-divm}, we have  
\beq\label{divbound}
\norm{\diver \uu_\varep}_{0,s} \le C_3 d_3(\varep,\beta,\gamma),\text{ where } 
d_3(\varep,\beta,\gamma)=\gamma +\varep d_2(\varep,\beta,\gamma)^{r-1}.
\eeq

Note that $d_i(\varep,\beta,\gamma)$, for $i=1,2,3$, are increasing in $\varep,\beta,\gamma$, and, by \eqref{piVZ},  $\beta,\gamma\le \|\mathcal F\|_{\mathcal W'}$.
Summing up the estimates \eqref{mLsbound}, \eqref{ub1} and  \eqref{divbound} gives
\beq
\norm{(\uu_\varep,p_\varep)}_{\mathcal W}
\le \mathcal C \eqdef \sum_{i=1}^3 C_i d_i(1,\norm{\mathcal F}_{\mathcal W'},\norm{\mathcal F}_{\mathcal W'}).
\eeq
Thus, we obtain the desired estimate \eqref{uQmV}.
\end{proof}

\subsection{Existence and uniqueness}\label{subeu}

\begin{theorem}\label{SolofStationaryProb}
For any $\mathcal F\in \mathcal W'$, equation \eqref{VQgen} has a unique solution $(\uu,p)\in \mathcal W$.  
\end{theorem}
\begin{proof}
Given $\mathcal F\in \mathcal W'$, let $\varep_0$ and $\mathcal C$ be as in Lemma \ref{stationary-sol-indep-eps}.
For $\varep\in(0,\varep_0)$, let $(\uu_\varep,p_\varep)$ be the unique solution of the regularized problem \eqref{dualeq}. 
Note from \eqref{a-eps}, \eqref{VQreg} and \eqref{RRineq} that
\begin{align*}
&\left|\widehat a((\uu_\varep,p_\varep),(\vv,q)) - \mathcal F(\vv,q)\right|
=\varep |R_1(\uu_\varep,\vv)+R_2(p_\varep,q)|\\
&\le \varep \big(\norm{\diver\uu_\varep}_{0,s}^{{ s - 1}}\norm{\diver\vv}_{0,s} + \norm{  p_\varep}_{0,r}^{r-1}\norm{q}_{0,r} \big)
\le \varep \big(\norm{\uu_\varep}_V^{s-1} + \norm{p_\varep}_Z^{r-1} \big)\norm{(\vv,q)}_{\mathcal W} .
\end{align*}
Hence 
\beq\label{Aconv}
\norm{\mathcal A(\uu_\varep,  p_\varep) -\mathcal F}_{\mathcal W'} 
\le \varep\big( \norm{\uu_\varep}_V^{s-1} +\norm{  p_\varep}_Z^{r-1}  \big)
\le 2\varep (1+\mathcal C)^{s-1}\to 0\text{ as } \varep\to 0.
\eeq 

Let $N>1$ such that $1/N<\varep_0$.
By the virtue of Lemma \ref{stationary-sol-indep-eps}, the sequence  $((\uu_{1/k},  p_{1/k}))_{k=N}^\infty$ is a bounded sequence in the reflexive Banach space $\mathcal W$. Then there exists a subsequence $((\uu_{1/k_j},  p_{1/k_j}))_{j=1}^\infty$ that converges weakly in $\mathcal W$ to an element  $(\uu, p)\in \mathcal W.$ 
Combining this fact with the limit \eqref{Aconv}, and applying a general analysis result \cite[part (c), p. 474]{ZeidlerIIB}, we obtain  $\mathcal A(\uu, p) = \mathcal F$. Therefore, $(\uu, p)$ is a solution of \eqref{VQgen}.

\medskip
We prove the uniqueness next. Suppose  $(\uu_1,   p_1)$ and  $(\uu_2,  p_2)$ are two solutions of \eqref{VQgen}. 
Subtracting their equations and choosing the testing functions $(\vv,q)=(\uu_1-\uu_2,  p_1-  p_2)$, we obtain 
\beqs
0=\intb{\mathcal A(\uu_1,p_1)-\mathcal A(\uu_2,p_2), (\uu_1-\uu_2,  p_1-  p_2) }_{\mathcal W',\mathcal W}=
a(\uu_1,\uu_1-\uu_2)-a(\uu_2,\uu_1-\uu_2).
\eeqs
Combining this with property \eqref{a4} yields $0\ge c_2 \norm{ \uu_1-\uu_2}_{0,s}^s$, which implies $\uu_1=\uu_2$.

For $i=1,2$, we have the variational equation 
$a(\uu_i,\vv)-b(\vv, p_i)= (\pi_1 \mathcal F)\vv$ for all $\vv\in V$. 
Subtracting these two equations implies $b(\vv,p_1-p_2)=a(\uu_1,\vv) - a(\uu_2,\vv)=0$  for all $\vv\in V$.
By applying inequality \eqref{supinfcdn} in Lemma~\ref{dualnorm} to $q=p_1-p_2$, we obtain $p_1=p_2$.  
\end{proof}

\subsection{Estimates and continuous dependence}\label{subcont}

Regarding the unique solutions of equation \eqref{VQgen}, we have the following estimates.

\begin{theorem}\label{postest}
Assume $\mathcal F\in \mathcal W'$, $(\uu,p)\in \mathcal W$,  and $\mathcal A(\uu,p)=\mathcal F$. 
Denote $\beta=\|\pi_1\mathcal F\|_{V'}$  and $\gamma=\|\pi_2\mathcal F\|_{Z'}$.
Then
\beq\label{udivz}
\norm{\uu}_{0,s}\le c_7(\beta^{r-1}+\gamma^{r-1}+\gamma),
 \quad \norm{\diver \uu}_{0,s} \le c_7\gamma,
 \quad \norm{ p}_Z
\le c_7( \gamma^{r-1}+\gamma^{s-1} + \beta^{r-1} + \beta),
\eeq
and, hence,
\beq\label{upW2}
\norm{(\uu,p)}_{\mathcal W}\le c_8\left ( \gamma^{r-1}+\gamma^{s-1} + \beta^{r-1} + \beta \right ).
\eeq
Consequently,
\beq\label{upW}
\norm{(\uu,p)}_{\mathcal W}\le c_9\left (\norm{\mathcal F}_{\mathcal W'}^{r-1}+\norm{\mathcal F}_{\mathcal W'}^{s-1}\right ).
\eeq 
Above, the positive constants $c_7$, $c_8$, $c_9$ are independent of $\mathcal F,\uu,p$.
\end{theorem}
\begin{proof}
We repeat the calculations in Lemma \ref{stationary-sol-indep-eps} with $\varep=0$.
It follows \eqref{mLsbound}, \eqref{ub1} and \eqref{divbound} that
\beq
\norm{\uu}_{0,s}\le Cd_1, \quad 
\norm{ p}_Z
\le Cd_2,\quad \norm{\diver \uu}_{0,s} \le Cd_3,
\eeq 
where  $d_1=(\beta^r +\beta \gamma+\gamma^r+ \gamma^{s})^{1/s}$, $d_2=d_1+d_1^{s-1}  +  \beta$ and $d_3= \gamma$. 

Let $C$ denote a generic positive constant as in the proof of Lemma \ref{stationary-sol-indep-eps}.
Using inequalities \eqref{ele0} and \eqref{Yineq} yields   
\begin{align*}
d_1&\le C\big (\beta^r +\gamma^r+ \gamma^{s}\big)^{1/s}\le C(\beta^{r-1}+\gamma^{r-1}+\gamma),\\
d_2&\le C\big(\beta^{r-1} +\gamma^{r-1}+ \gamma\big)
+C\big (\beta^{r-1} +\gamma^{r-1}+ \gamma \big)^{s-1} 
+  \beta
\le C(\beta^{r-1} + \beta+ \gamma^{r-1}+\gamma^{s-1}  ).
\end{align*}
Therefore, we obtain the estimates in  \eqref{udivz}.
Summing up the estimates in \eqref{udivz}, noticing that $r-1<1<s-1$ and applying  inequality \eqref{Yineq} yield \eqref{upW2}.

Finally, inequality \eqref{upW} follows \eqref{upW2}, the fact $\beta,\gamma\le \|\mathcal F\|_{\mathcal W'}$, which is due to  \eqref{piVZ},  and, again, the use of \eqref{Yineq}.  
\end{proof}

It turns out that the solutions of \eqref{VQgen} depend continuously on $\mathcal F$, as showed below.

\begin{theorem}\label{cdthm}
For $j=1,2$, let $\mathcal F_j\in\mathcal W'$ and $(\uu_j,p_j)\in \mathcal W$ satisfy $\mathcal A(\uu_j,p_j)=\mathcal F_j$.
Denote 
\beq\label{Mbg}
M=1+\sum_{j=1}^2 \left (\norm{\pi_1 \mathcal F_j}_{V'}^{r-1}+\norm{\pi_2 \mathcal F_j}_{Z'}\right ),\quad 
\beta=\norm{\pi_1 (\mathcal F_1-\mathcal F_2)}_{V'}\text{ and }
\gamma=\norm{\pi_2 (\mathcal F_1-\mathcal F_2)}_{Z'}.
\eeq 
Then 
\beq\label{cde}
\norm{(\uu_1,p_1)-(\uu_2,p_2)}_{\mathcal W}\le  c_{10}(\beta + M^{s-2}\beta^{r-1} +\gamma+  M^{r(s-2)}\gamma^{r-1} ).
\eeq
Consequently,
\beq\label{cdeF}
\norm{(\uu_1,p_1)-(\uu_2,p_2)}_{\mathcal W}
\le c_{11} \left (1+\norm{\mathcal F_1}_{\mathcal W'}+\norm{\mathcal F_2}_{\mathcal W'}\right )^{r(s-2)}
(\norm{\mathcal F_1-\mathcal F_2}_{\mathcal W'}^{r-1}+\norm{\mathcal F_1-\mathcal F_2}_{\mathcal W'}).
\eeq
Above positive constants $c_{10}$ and $c_{11}$ are independent of $\mathcal F_j,\uu_j,p_j$ for $j=1,2$.
\end{theorem}
\begin{proof} 
Below, $C$ denotes a generic positive constant independent of $\mathcal F_j,\uu_j,p_j$ for $j=1,2$.

Denote $\uu=\uu_1-\uu_2$,  $p=p_1-p_2$, and $\mathcal F = \mathcal F_1 -\mathcal F_2$.  
Then $\beta=\norm{\pi_1 \mathcal F}_{V'}$, $\gamma=\norm{\pi_2 \mathcal F}_{Z'}$ and,
thanks to the first estimate in \eqref{udivz} for $\uu_1$ and $\uu_2$, 
\beq \label{uM}
\norm{\uu_1}_{0,s}+\norm{\uu_2}_{0,s}\le CM.
\eeq

For $j=1,2$, one has 
\begin{align} 
a(\uu_j,\vv)-b(\vv, p_j)&= (\pi_1 \mathcal F_j)(\vv)  && \text{ for all }\vv\in V,\\
b(\uu_j, q)&=  (\pi_2 \mathcal F_j)(q)  && \text{ for all } q\in Z.
\end{align}
It follows that 
\begin{align}\label{cd1} a(\uu_1,\vv)-a(\uu_2,\vv)-b(\vv, p)&= (\pi_1 \mathcal F)(\vv) &&\text{ for all }\vv\in V,\\ 
\label{cd2}
b(\uu, q)&=  (\pi_2 \mathcal F)(q)&& \text{ for all } q\in Z.
\end{align}
Taking $q=|\diver \uu|^{s-2}\diver \uu$ in \eqref{cd2} gives
$
\|\diver \uu\|_{0,s}^s\le \gamma \|\diver \uu\|_{0,s}^{s-1}.
$
Thus,
\beq\label{cd3}
\|\diver \uu\|_{0,s}\le \gamma .
\eeq

For any $\vv\in V$, we have from \eqref{cd1} and \eqref{a3} that 
\beqs 
b(\vv, p)
\le |(\pi_1 \mathcal F)\vv|+|a(\uu_1,\vv)-a(\uu_2,\vv)|
\le \beta  \|\vv\|_V + c_6 \left(1+\norm{\uu_1}_{0,s}^{s-2}+ \norm{\uu_2}_{0,s}^{s-2}\right)\norm{\uu}_{0,s}\|\vv\|_V.
\eeqs  
Hence, making use of Lemma \ref{dualnorm} to estimate $\|p\|_{Z}$ and also \eqref{uM} gives
\beq\label{cd4}
\|p\|_Z\le C (\beta+M^{s-2}\norm{\uu}_{0,s}).
\eeq

Taking $\vv=\uu$ in  \eqref{cd1}, 
 $q=p$ in \eqref{cd2},  and  adding the resulting equations yield 
$$
a(\uu_1,\uu)-a(\uu_2,\uu)=(\pi_1 \mathcal F)\uu + (\pi_2 \mathcal F)p .
$$
Applying inequality \eqref{a4} to the left-hand side and  inequality \eqref{mFuv} to the right-hand side, we derive 
\beq\label{ineq-u}
\norm{\uu}_{0,s}^s
    \le C(\beta \|\uu\|_V +\gamma \|p\|_Z) \le C(\beta(\|\uu\|_{0,s}+\|\diver \uu \|_{0,s}) +\gamma \|p\|_Z) .
\eeq
Inserting \eqref{cd3} and \eqref{cd4} into \eqref{ineq-u} and using Young's inequality give   
\begin{align*}
    \norm{\uu}_{0,s}^s
    \le C(\beta + M^{s-2}\gamma)\norm{\uu}_{0,s} +C \beta \gamma
    \le \frac12 \norm{\uu}_{0,s}^s + C(\beta^r+ M^{r(s-2)}\gamma^r + \beta\gamma).
\end{align*}
Thus,
\beq \label{cd5} 
\norm{\uu}_{0,s} 
\le C (\beta^{r-1}+ M^{(r-1)(s-2)}\gamma^{r-1} + \beta^{1/s}\gamma^{1/s}).
\eeq 
Utilizing this estimate in \eqref{cd4}   gives
\beq
\begin{split}
\|p\|_Z&\le C (\beta + M^{r(s-2)}\gamma^{r-1} + M^{s-2}(\beta^{r-1}+  \beta^{1/s}\gamma^{1/s})).
\end{split}
\eeq
By Young's inequality again, 
$ M^{s-2}\beta^{1/s}\gamma^{1/s} \le \beta +M^{r(s-2)}\gamma^{r-1}. $
We obtain
\beq\label{cd6}\begin{split}
\|p\|_Z&\le C\beta + C M^{r(s-2)}\gamma^{r-1} + C M^{s-2}\beta^{r-1}.
\end{split}
\eeq

Combining \eqref{cd3}, \eqref{cd5}, \eqref{cd6} gives
\beq\label{cd7}
\norm{(\uu,p)}_{\mathcal W}\le  C\gamma+ C M^{r(s-2)}\gamma^{r-1} + C\beta + C M^{s-2}\beta^{r-1} + C\beta^{1/s}\gamma^{1/s}.
\eeq
By Young's inequality, $\beta^{1/s}\gamma^{1/s}\le \beta +\gamma^{r-1}$. Hence, we obtain \eqref{cde} from\eqref{cd7}.

Finally, using the facts $\beta,\gamma\le \norm{\mathcal F}_{\mathcal W'}$ and $M\le 3(1+\norm{\mathcal F_1}_{\mathcal W'}+\norm{\mathcal F_2}_{\mathcal W'})$, we obtain \eqref{cdeF} from \eqref{cde}.
\end{proof}

 \subsection{Proof of Theorem \ref{mainthm}}\label{mainpf}
 \begin{proof}
 \medskip
Part~\ref{t1}. The statement follows Theorem \ref{SolofStationaryProb} when applied to $\mathcal F=\mathcal F_{\psi,f}$.  

\medskip
Part~\ref{t2}. We apply Theorem \ref{postest} to $\mathcal F=\mathcal F_{\psi,f}$.
Note from \eqref{ppsf}, \eqref{Phinorm} and \eqref{VXnorms} that
$$\beta=\norm{\mathcal T_\psi}_{V'}\le \bar {c}_1\norm{\psi}_{X^r} \text{ and }\gamma=\norm{\Phi_f}_{Z'}=\norm{f}_{0,s}.$$
Then estimate \eqref{udivp0} follows \eqref{upW2}.

\medskip
Part~\ref{t3}. We apply Theorem \ref{cdthm} to $\mathcal F_j=\mathcal F_{\psi_j,f_j}$ for $j=1,2$.
Using properties \eqref{ppsf} and \eqref{Phinorm}, \eqref{VXnorms}  we have $M$, $\beta$ and $\gamma$ defined in \eqref{Mbg} satisfy
$\beta\le \bar c_1\|\psi_1-\psi_2\|_{X_r}$, $\gamma=\|f_1-f_2\|_{0,s}$, and, with the observation $(r-1)(s-2)=2-r$,
$$M^{s-2}\le C\left (1+ \sum_{j=1}^2 (\norm{\psi_j}_{X_r}^{2-r}+\norm{f_j}_{0,s}^{s-2})\right)=CM_0,$$
for some positive constant $C$.
Then estimate \eqref{contpf1} follows \eqref{cde}. Recalling $2-r=(r-1)(s-2)<s-2$ and applying Young's inequality, we obtain \eqref{contpf2} from \eqref{contpf1}.
\end{proof}

\section{Conclusions}\label{conclude}
In this paper, we have analyzed  the Forchheimer-typed flows of compressible fluids
 in anisotropic porous media. By combining the anisotropic momentum equation of the Forcheimer type with the conservation of mass, we obtained a system of nonlinear partial differential equations in the case of isentropic flows and slightly compressible flows.  
Many sufficient conditions for the important monotonicity property were derived for different models. 
Based on this crucial understanding of the structure of the momentum equation, we studied the steady state flows subject to a nonhomogeneous Dirichlet boundary condition. We proved the existence and uniqueness of the weak solutions, derived their bounds in terms of  the forcing  function and the boundary data, and established their continuous dependence on those function and data. 
 However, similar results for time-dependent solutions have not been achieved. It turns out that the anisotropic properties pose much difficulty in applying the known techniques. 
 For example, the straight forward technique of semi-discretization in time seemingly fails in obtaining the existence of the time-dependent solutions. Further study is certainly needed.
From the practical point of view, the numerical methods for both stationary and evolutionary equations need to be developed which must  overcome the mathematical challenge of the anisotropic conditions.

\appendix 
\section{}\label{apA}

\begin{lemma}\label{Wreflex}
For any $s\in(1,\infty)$, $\mathbf W_s(\rm{div},\Omega)$ is a reflexive Banach space.    
\end{lemma}
\begin{proof}
Firstly, one can verify that $\mathbf W_s(\rm{div},\Omega)$ is a Banach space.
Same as in \cite{KS16}, we use the mapping $E\vv=(\vv,\diver \vv)$ for $\vv\in \mathbf W_s(\rm{div},\Omega)$ to embed $\mathbf W_s(\rm{div},\Omega)$ into $(L^s(\Omega))^{n+1}$. Denote $\widetilde{\mathbf W}_s=E(\mathbf W_s(\rm{div},\Omega))\subset  (L^s(\Omega))^{n+1}$.
Then the norm $\norm{E\vv}_{(L^s(\Omega))^{n+1}}$ in $\widetilde{\mathbf W}_s$ is equivalent to the norm $\norm{\vv}_V$.
This way, we can identify $\mathbf W_s(\rm{div},\Omega)$ as $\widetilde{\mathbf W}_s$ and vice versa.
As a consequence, $\widetilde{\mathbf W}_s$ is a closed subspace of $(L^s(\Omega))^{n+1}$, hence, it is a reflexive Banach space.

For dual and double dual spaces, we identify $F\in \mathbf W_s(\rm{div},\Omega)'$ as $\widetilde F=F\circ E^{-1}\in \widetilde{\mathbf W}_s'$, and identify
$G\in \mathbf W_s(\rm{div},\Omega)''$ as $\widetilde G\in \widetilde{\mathbf W}_s''$ defined by
$\widetilde G(\widetilde F)=G(\widetilde F\circ E)$ for any  $\widetilde F\in \widetilde{\mathbf W}_s'$.
Then $\widetilde{\mathbf W}_s$  being reflexive implies that $\mathbf W_s(\rm{div},\Omega)$  is a reflexive Banach space.
\end{proof}

\begin{proof}[Proof of inequality \eqref{ele2}]
We call Scenario 1 the case when the origin is not on the line segment  connecting $x$ and $y$, and Scenario 2 otherwise.
Consider Scenario 1 first. 
Let $\gamma(t)=t x+ (1-t) y\ne 0$ for $t\in [0,1]$ and $h_1(t) =|\gamma(t)|^p$.   Then 
\begin{align*}
||x|^p - |y|^p |&=\left|\int_0^1 h_1'(t) \d t\right|
 \le \int_0^1 p|\gamma(t)|^{p-2} |\gamma(t)\cdot (x-y)|  \d t
\le  p|x-y| \int_0^1 |\gamma(t)|^{p-1}  \d t\\
&\le p|x-y|\int_0^1 2^{(p-2)^+} (t^{p-1}|x|^{p-1}+ (1-t)^{p-1}|y|^{p-1})  \d t\\
&= 2^{(p-2)^+} (|x|^{p-1}+ |y|^{p-1})|x-y|,
\end{align*}
which yields \eqref{ele2}.
In Scenario 2, 
we have $y = -k x$ or $x=-ky$ for some  number $k\in[0,\infty)$. Without loss of generality, assume the former situation.
Then we have 
 \beqs    
||x|^p - |y|^p |= |x|^p |1-k^p |\le |x|^p (1+k^p)\le |x|^p (1+k^{p-1})(1+k)=(|x|^{p-1}+ |y|^{p-1})|x-y|, \quad 
 \eeqs
which shows  \eqref{ele2} again.
\end{proof}

\begin{proof}[Proof of Lemma \ref{ILem1}]
Consider $p>0$ and $x,y\in\R^n$.
We use the same named scenarios and the function $\gamma(t)$ as in the proof of inequality \eqref{ele2}.

\medskip\noindent
\emph{Proof of  inequality \eqref{inq1}.}    
Consider Scenario 1 and define $h_2(t) =|\gamma(t)|^p \gamma(t)$ for $t\in[0,1]$.   Then 
\begin{align*}
||x|^p x- |y|^p y|&=\left|\int_0^1 h_2'(t) \d t\right|
 =\left|\int_0^1 |\gamma(t)|^p(x-y) + p|\gamma(t)|^{p-2} (\gamma(t)\cdot (x-y))\gamma(t)    \d t \right| \\ 
&\le  (1+p)|x-y| \int_0^1 |\gamma(t)|^p  \d t
\le (1+p)|x-y|\int_0^1  2^{(p-1)^+}(t^p|x|^p+(1-t)^p |y|^p)  \d t\\
&=2^{(p-1)^+} (|x|^p+ |y|^p)|x-y|.
\end{align*}
This proves \eqref{inq1}.
In Scenario 2, we can assume $y = -k x$ for some $k\ge 0$. We have  
 \beqs    
 ||x|^{p} x-|y|^{p} y|=|x|^{p+1} (1+k^{p+1}) \le |x|^{p+1} (1+k^p)(1+k)
 =(|x|^p+ |y|^p)|x-y|.
 \eeqs
 Hence, we obtain \eqref{inq1}.
 
\medskip\noindent
\emph{Proofs of  inequalities \eqref{inq3m} and \eqref{inq3}.}    Letting $z=x-y$, we have
    \begin{align*}
        (|x|^p x-|y|^p y)\cdot (x-y)
        &=\left (|x|^p\big(\frac{x+y}2+\frac z 2\big)-|y|^p \big(\frac{x+y}2-\frac z 2\big)\right )\cdot z \\
        &=(|x|^p-|y|^p)\frac{x+y}2\cdot z +\frac12(|x|^p+|y|^p)|z|^2\\
        &=\frac12(|x|^p-|y|^p)(|x|^2-|y|^2)+\frac12(|x|^p+|y|^p)|z|^2.
    \end{align*}
Since $(|x|^p-|y|^p)(|x|^2-|y|^2)\ge 0$, we obtain \eqref{inq3m}. Using $(|x|+|y|)^p\ge 2^{-(p-1)^+}|x-y|^p $, we  then deduce \eqref{inq3} from \eqref{inq3m}.

\medskip

Now, consider $p\in(-1,0)$.

\medskip\noindent
\emph{Proof of inequality \eqref{inq2}.}  Let $x,y\in\R$.  The inequality obviously holds true when $x=0$ or $y=0$. Also, by switching the roles of $x$ and $y$, we can assume $x>0$ and $y\ne 0$.

If $y>0$, then $||x|^{p} x-|y|^{p} y|=|x^{1+p}-y^{1+p}|$. Noting that $1+p\in(0,1)$, we apply inequality \eqref{ele1} to have
\beqs
||x|^{p} x-|y|^{p} y|\le |x-y|^{1+p}\le 2^{-p}|x-y|^{1+p}.
\eeqs

If $y<0$, then
$
||x|^{p} x-|y|^{p} y|=|x|^{1+p}+|y|^{1+p}.
$
Applying H\"older's inequality to the dot product of two vectors $(|x|^{1+p},|y|^{1+p})$ and $(1,1)$ with powers $1/(1+p)$ and $-1/p$,  we obtain 
\beqs
||x|^{p} x-|y|^{p} y|\le (|x|+|y|)^{1+p}\cdot 2^{-p} =2^{-p}|x-y|^{1+p},
\eeqs
which yields \eqref{inq2}.

\medskip\noindent
\emph{Proof of inequality \eqref{inq4}.}    Let $x,y\in\R^n$.
Consider Scenario 1 and define the function 
$$h_3(t) =|\gamma(t)|^{p} \gamma(t)\cdot (x-y) \text{ for }t\in[0,1].$$
Then 
\begin{align*}
(|x|^{p} x- |y|^{p} y)\cdot(x-y)&=\int_0^1 h_3'(t) \d t
=\int_0^1|\gamma(t)|^{p} |x-y|^2 +p |\gamma(t)|^{p-2}|\gamma(t)\cdot (x-y)|^2 d t\\
&\ge (1+p)|x-y|^2\int_0^1 |\gamma(t)|^{p} \d t.
\end{align*}
Note that $-p\in(0,1)$, hence 
$|\gamma(t)|^{-p} \le (|x|+|y|)^{-p}$.
Therefore, we obtain \eqref{inq4}.

In Scenario 2,  we can assume $y = -k x$ for some $k\ge 0$. We have   
\begin{align*}
(|x|^{p} x- |y|^{p} y)\cdot(x-y)
=|x|^{2+p} (1+k^{1+p})(1+k).
\end{align*}
Since $0<1+p<1$, we have from \eqref{ele0} that $1+ k^{1+p}\ge (1+k)^{1+p}$. Hence,
\begin{align*}
(|x|^{p} x- |y|^{p} y)\cdot(x-y)
\ge |x|^{2+p}(1+k)^{2+p}
= |x-y|^2 (|x|+|y|)^{p},
\end{align*}
which proves \eqref{inq4} again.
\end{proof}

\begin{proof}[Proof of Lemma \ref{dualnorm}]
Let $q\in L^r(\Omega)$. If $q=0$ then \eqref{supinfcdn} holds true. Consider $q\ne 0$.
Denote by $W^{1,r}_0(\Omega)$ the space of functions in $W^{1,r}(\Omega)$ having zero trace on the boundary.
Note that $|q|^{r-2}q\in L^s(\Omega)$. By the Browder--Minty Theorem, there exists a unique solution $w\in W^{1,s}_0(\Omega)$ of the problem
\beq\label{Bpb}
\int_\Omega |\nabla w|^{r-2}\nabla w\cdot\nabla v \d x=\int_\Omega |q|^{r-2}q v \d x\quad\text{ for all }v\in W_0^{1,r}(\Omega). 
\eeq
Choosing $v=w$ in \eqref{Bpb} and applying the H\"older and Poincar\'e inequalities give 
\beqs
\norm{\nabla w}_{0,r}^r=\int_\Omega  |\nabla w|^r dx = \int_\Omega |q|^{r-2} q w dx
\le \norm{q}_{0,r}^{r-1} \norm{w}_{0,r}\le C \norm{q}_{0,r}^{r-1} \norm{\nabla w}_{0,r},
\eeqs
where $C$ is a positive constant. Hereafter, $C$ denotes a generic positive constant.
It follows that
\beqs
\norm{\nabla w}_{0,r}\le C\norm{q}_{0,r}.
\eeqs 
Set $\uu=-|\nabla w|^{r-2}\nabla w$. Then $\uu\in \mathbf L^s(\Omega)$ and, by \eqref{Bpb}, $\diver \uu=|q|^{r-2}q\in L^s(\Omega)$.
Therefore, $\uu\in V\setminus\{0\}$. Observe that
\begin{align*}
\norm{\uu}_V^s &= \norm{\uu}_{0,s}^s  +\norm{\diver \uu}_{0,s}^s
=\norm{\nabla w}_{0,r}^r + \norm{q}_{0,r}^r 
\le C \norm{q}_{0,r}^r.
\end{align*}
Thus, $\norm{\uu}_V \le C \norm{q}_{0,r}^{r-1} $. We then have
$$\int_\Omega (\diver  \uu) q dx = \int_\Omega |q|^r dx = \norm{q}_{0,r}^r 
=\norm{q}_{0,r} \norm{q}_{0,r}^{r-1}
\ge C\norm{q}_{0,r} \norm{\uu}_V. $$
Consequently, we obtain inequality \eqref{supinfcdn}.
\end{proof}

\medskip
\noindent\textbf{Acknowledgement.} 
The authors would like to thank the referees for many valuable suggestions that help improve the paper.

\medskip
\noindent\textbf{Data availability.} 
No new data were created or analyzed in this study.

\medskip
\noindent\textbf{Funding.} No funds were received for conducting this study. 

\medskip
\noindent\textbf{Conflict of interest.}
There are no conflicts of interests.

\bibliography{paperbaseall}{}
\bibliographystyle{abbrv}

\end{document}